\documentclass[12pt,twoside,reqno,psamsfonts]{amsart}

\usepackage{graphicx}
\usepackage{amsmath}
\usepackage{amssymb}
\usepackage{enumerate}
\usepackage{verbatim}
\usepackage{url}

\numberwithin{equation}{section}
\newcommand{\R}{\mathbb{R}}
\newcommand{\C}{\mathbb{C}}

\newcommand{\N}{\mathbb{N}}

\newcommand{\cL}{\mathcal{L}}

\newcommand{\cZ}{\mathcal{Z}}

\newcommand{\e}{\mathbf{e}}
\renewcommand{\d}{\mathbf{d}}

\newcommand{\cR}{\mathcal{R}}

\newcommand{\cW}{\mathcal{W}}

\newcommand{\cP}{\mathcal{P}}

\newcommand{\cA}{\mathcal{A}}

\newcommand{\eps}{\varepsilon}

\newcommand{\Hd}{\dim_\mathrm{H}}

\newcommand{\Z}{\mathbb{Z}}

\renewcommand{\emptyset}{\varnothing}
\renewcommand{\epsilon}{\varepsilon}
\renewcommand{\rho}{\varrho}
\renewcommand{\phi}{\varphi}

\renewcommand{\mod}{\,\,\mathrm{mod\,}}

\renewcommand{\a}{\mathbf{a}}
\renewcommand{\b}{\mathbf{b}}
\renewcommand{\c}{\mathbf{c}}
\newcommand{\w}{\mathbf{w}}

\newcommand{\A}{\mathbf{A}}
\newcommand{\B}{\mathbf{B}}

\renewcommand{\hat}{\widehat}

\renewcommand{\iint}{\int\hspace{-0.1in}\int}

\DeclareMathOperator{\supp}{supp}

\theoremstyle{plain}
\newtheorem{thm}{Theorem}[section]
\newtheorem{theorem}{Theorem}[section]

\newtheorem{lemma}[thm]{Lemma}
\newtheorem{prop}[thm]{Proposition}

\theoremstyle{definition}

\newtheorem{definition}[thm]{Definition}

\newtheorem{question}[thm]{Question}

\theoremstyle{remark}
\newtheorem{remark}[thm]{Remark}

\usepackage[usenames,dvipsnames]{xcolor}

\graphicspath{ {pictures/} }

\begin{document}

\title{Fourier transform and expanding maps on Cantor sets}

\author{Tuomas Sahlsten}
\email{tuomas.sahlsten@aalto.fi}
\address{Department of Mathematics and Systems Analysis, Aalto University, Finland}

\author{Connor Stevens}
\email{connor.stevens@manchester.ac.uk}
\address{Department of Mathematics, University of Manchester, UK}

\thanks{TS was supported by a start-up fund from the University of Manchester. CS was supported by the Doctoral Programme in the University of Manchester and EPSRC}

\maketitle

\begin{abstract}
We study the Fourier transforms $\widehat{\mu}(\xi)$ of non-atomic Gibbs measures $\mu$ for uniformly expanding maps $T$ of bounded distortions on $[0,1]$ or Cantor sets with strong separation. When $T$ is totally non-linear, then $\widehat{\mu}(\xi) \to 0$ at a polynomial rate as $|\xi| \to \infty$. 
\end{abstract}


\section{Introduction}

Given a Borel measure $\mu$ on $\R$ and a frequency $\xi \in \R$, then the corresponding Fourier coefficient (or amplitude in frequency $\xi$) associated to $\mu$ is given by the complex number
$$\widehat{\mu}(\xi) = \int e^{-2\pi i \xi x} \, d\mu(x)$$ 
for $\xi \in \R.$ The Fourier coefficients of $\mu$ relate closely to various fine structure properties of the measure. For example, the Riemann-Lebesgue lemma states that if $\mu$ is absolutely continuous with $L^1$ density, then $\widehat{\mu}(\xi)$ converges to $0$ when the frequencies $|\xi| \to \infty$; we call such $\mu$ \textit{Rajchman measures} \cite{Lyons}. In contrast for atomic measures $\mu$, Wiener's theorem \cite{Lyons} says that their Fourier transform $\widehat{\mu}(\xi)$ cannot converge to $0$ as $|\xi| \to \infty$. The intermediate case, namely, fractal measures has a great deal of history. For example, it relates to classifying the sets of uniqueness and multiplicity for trigonometric series \cite{Cohen,K,JialunSahlsten1}, Diophantine approximation \cite{PVZZ}, Fourier dimension \cite{Mattila15,EkstromSchmeling} and $L^p$ multiplier theory \cite{Sarnak,SidorovSolomyak,Stein,Zygmund}. For the middle-third Cantor measure, the Fourier transform cannot decay at infinity due to invariance under $\times 3 \mod 1$, but some other fractal measures such as random measures (see Salem's work \cite{Salem} on random Cantor measures or Kahane's work on Brownian motion \cite{KahaneImage,KahaneLevel}), measures arising in Diophantine approximation (see Kaufman's measures on well and badly approximable numbers \cite{Kauf2overalpha,Kaufman}), non-Pisot Bernoulli convolutions \cite{Salem}, and self-similar measures with suitable irrationality properties \cite{JialunSahlsten1, Bremont, VarjuYu} all exhibit decay of Fourier coefficients.

In a random setting, the conditions for Fourier decay usually require certain rapid correlation decay properties of the processes such as independent increments on the Brownian motion (see Kahane's work \cite{KahaneImage,KahaneLevel}), or other independence or Markov properties (see the works of Shmerkin and Suomala \cite{ShmerkinSuomala}). In the deterministic setting, the known examples are currently suggesting some form of nonlinearity starting from the works of Kaufman \cite{Kaufman,Kauf2overalpha}, where Rajchman measures were constructed on sets of well and badly approximable numbers. Such sets are naturally constructed using the \textit{Gauss map} defined by
$$G : x \mapsto \frac{1}{x} \mod 1, \quad x \in (0,1], \quad G(0) = 0$$

The Gauss map forms a crucial dynamical system in the theory of Diophantine approximation as it can be used to generate continued fraction expansions, and the geodesic flow on modular surface can be connected to its evolution by suspension flows \cite{BowenSeries}. In contrast to the $\times 3$ map, which has fully linear inverse branches, the Gauss map exhibits nonlinear inverse branches. Any $\times 3$ invariant measure cannot have Fourier decay, but as proven by Jordan and the first author \cite{JordanSahlsten} in 2013, when assuming certain correlation properties from the invariant measures for the Gauss map (Bernoulli or more generally Gibbs property) and finite Lyapunov exponent, then invariant measures of large enough dimensions exhibit Fourier decay. However, the proof in \cite{JordanSahlsten} which was based on ideas of Kaufman \cite{Kaufman} and Queff\'elec and Ramar\'e \cite{QetR}, was based heavily on Diophantine approximation, in particular the distributional properties of continuants \cite{QetR}. Hence it would be interesting to know what  properties of general dynamical systems could replace these Diophantine results used in \cite{JordanSahlsten}.

The story progressed in 2017 when Bourgain and Dyatlov \cite{bourg} adapted the discretised sum-product theory from additive combinatorics developed by Bourgain \cite{Bourgain2010} and proved Fourier decay of Patterson-Sullivan measures for convex cocompact Fuchsian groups. This was also proved by Li \cite{Jialun} using different renewal theoretic tools from random walks on matrix groups. Patterson-Sullivan measures are self-conformal measures associated to an iterated function system given by contractive fractional linear transformations
$$x \mapsto \frac{ax+b}{cx+d}, \quad x \in [0,1]$$
with $ad-bc = 1$ and $a,b,c,d \in \R$ chosen such that the map is a contraction. This reflects the situation of the Gauss map, where the inverse branches of the map are of the form
$$x \mapsto \frac{1}{x+a}, \quad x \in [0,1]$$
for $a \in \N$, so the work of Bourgain and Dyatlov generalises \cite{JordanSahlsten} to more general fractional linear transformations. The motivation for the results of Bourgain and Dyatlov \cite{bourg} is to establish a \textit{Fractal Uncertainty Principle} for the limit sets of Fuchsian groups. Fractal Uncertainty Principles, which have a few different forms, were introduced by Dyatlov and Zahl \cite{DyatlovZahl} as a powerful harmonic analytic tool to understand for example Pollicott-Ruelle resonances in open dynamical systems \cite{BourgainDyatlov}, and delocalisation of semiclassical limits of eigenfunctions for the Laplacian in Quantum Chaos \cite{DJ17}. We refer to the survey of Dyatlov \cite{Dyatlov} on more about the history and ideas behind Fractal Uncertainty Principles.

In the study of dimension theory for equilibrium states of fractional linear transformations, one is often able to generalise the results to equilibrium states $\mu$ associated to general expanding interval maps $T$ with enough regularity or distortion assumptions on the inverse branches of $T$ \cite{Keller}. However, recall that Fourier decay is not possible for fractal invariant measures of the interval map $T(x) = 3x \mod 1$ (e.g. for the middle-third Cantor measure), so some conditions on $T$ are required. The works \cite{JordanSahlsten,bourg} both share that the corresponding interval map $T$ is defined by \textit{non-linear} inverse branches. However, in order to have Fourier decay for $\mu$, one needs some type of \textit{non-concentration} of derivatives of $T$ along different orbits. Controlling the non-concentration of derivatives requires structure from the measure, which is missing from the middle-third Cantor set for example, but is present for Patterson-Sullivan measures of Fuchsian groups using the Fuchsian group structure \cite{bourg}, or for equilibrium states associated to the Gauss map using metric number theory and Diophantine approximation methods \cite{JordanSahlsten,QetR}. In our work we will find out that the necessary non-concentration is closely related to the cohomological properties of the \textit{distortion function} 
$$\tau = \log |T'|.$$

In order to state our main result formally, let us introduce some notation. We will follow a similar notation as Naud \cite{Naud}, but allow the intervals to meet at the boundary. Let $I_1,\dots,I_N$, $N \geq 2$, be closed intervals in $[0,1]$, with disjoint interiors. Let $T : I = \bigcup_{a = 1}^N I_a \to [0,1]$ be a mapping such that each restriction $T_a := T|_{I_a}$ is a $C^2$ diffeomorphism and assume $T$ is conjugated to the full shift on $\cA^\N$, where $\cA = \{1,\dots,N\}$. We say $T$ is \textit{uniformly expanding Markov map of bounded distortions} if it satisfies the conditions:
\begin{itemize}
\item[(1)] \textit{Uniform expansion}: There exists $\gamma > 1$ and $D > 0$ such that for all $n \in \N$ and all $x \in  I$ we have
$$|(T^n)'(x)| \geq D^{-1} \gamma^n.$$
\item[(2)] \textit{Markov property}: For all $a,b = 1,\dots, N$, if $T(I_b) \cap \mathrm{Int}(I_a) \neq \emptyset$, then $T(I_b) \supset I_a$.
\item[(3)] \textit{Bounded distortions}: For the distortion function $\tau = \log |T'|$, there there exists $B < \infty$ such that
 	$$ \|\tau'\|_\infty = \|T'' / T' \|_\infty \leq B.$$
\end{itemize}

Write $K = \bigcap_{ n = 0}^\infty T^{-n} ([0,1])$. If the intervals $I_a$ are disjoint, then $K$ is a Cantor set, if fact, that $K$ is the attractor to an iterated function system $\Phi := \{f_a : a \in \cA\}$, so conditions (1)-(3) are about the properties of the contractions $f_a$. In the literature \cite{PRKS} on sometimes calls $K$ a \textit{self-conformal set} and $\{f_a : a \in \cA\}$ a \textit{self-conformal iterated function system}.

For a given negative continuous potential $\phi : I \to \R$ such that the \textit{variations }
$$\mathrm{var}_n(\phi) =  \sup\{|\phi(f_{\a} (x)) - \phi(f_{\a} (y))| : x,y \in I, \a \in \cA^n\} = O(\rho^n),$$ 
as $n \to \infty$ for some $0 < \rho < 1$, we are interested in the measure $\mu = \mu_\phi$ on $K$ that realises the variational formula for the pressure:
$$P(\phi) = \sup\Big\{h_\mu(T) + \int_K \phi \, d\mu : \mu = T\mu \Big\}$$
where $h_\mu(T)$ is the entropy of $T$ with respect to $\mu$, which is called an \textit{equilibrium measure} (or \textit{Gibbs measure}), see Section \ref{sec:thermo} for more details and \cite{Keller} for more background. Note that under the assumptions (1)-(3), we know that \cite{Keller} the Hausdorff dimension of an equilibrium state $\mu$ is given by $\Hd \mu = h_\mu(T) / \lambda_\mu(T)$, where $\lambda_\mu(T) = \int \tau \, d\mu$ is the \textit{Lyapunov exponent} of $\mu$. Furthermore, let $s_0 > 0$ be the unique real solution to the equation $P(-s \tau) = 0$ for the distortion function $\tau = \log |T'|$. Then, if $K$ is a Cantor set, Bowen's formula \cite{Keller} gives that the Hausdorff dimension of $K$ is $\Hd K = s_0$. In the language of iterated function systems \cite{PRKS}, one also calls these Gibbs measures \textit{self-conformal measures} on the self-conformal set $K$. 

For the Fourier decay of $\mu_\phi$, the analysis will split into two cases on whether the dynamics of $T$ can be conjugated to a system with linear inverse branches or not. This can be formalised using the cohomological notion of total non-linearity of $T$:

\begin{itemize}
\item[(4)] \textit{Total non-linearity}:  $\tau = \log |T'|$ is \textit{not $C^1$ cohomologous to a locally constant function}: it is not possible to write
$$\tau = \psi_0 + g \circ T - g$$
on $I = \bigcup_{a = 1}^N I_a$ for $\psi_0 : I \to \R$ constant on every $I_a$, $a \in \cA$, and $g \in C^1(I)$.
\end{itemize}

In the language of analytic IFSs, the total non-linearity condition (4) follows if the IFS $\Phi = \{f_a : a \in \cA\}$ is \textit{not $C^2$ conjugated to a self-similar iterated function system}: there does not exist a $C^2$ diffeomorphism $h : \R \to \R$ such that the iterated function system $h\Phi = \{h f_a h^{-1} : a \in \cA\}$ consists of similitudes $x \mapsto r_a x + b_a$, $a \in \cA$, $0 < r_a < 1$, $b_a \in \R$. 

Total non-linearity assumption goes back to the uniform non-integrability properties of the unstable and stable foliations of Anosov flows and their symbolic properties, see for example Dolgopyat's works \cite{Dolgopyat1,Dolgopyat} and it is satisfied for example for various quadratic Julia sets and for limit sets of Schottky groups \cite{Naud}. In $\bigcup_{a = 1}^N I_a = [0,1]$ case, $K = [0,1]$, so total non-linearity is equivalent to the \textit{Uniform Non-Integrability} (UNI) of the distortion function $\tau$ (see the discussion later in Section \ref{sec:nonconcentration} below for a definition), which is proved in Avila, Gou\"ezel and Yoccoz \cite[Proposition 7.5]{AGY} on uniformly expanding maps on John domains in their study of exponential mixing of the Teichm\"uller flow. If the intervals $I_a$, $a \in \cA$, are disjoint, and one further assumes $f_a$ are analytic, then total non-linearity is equivalent to UNI, which follows by investigating the proof of \cite[Proposition 7.5]{AGY} or Naud's work \cite[Lemma 4.3 and Proposition 5.1]{Naud}. This equivalence goes back to the notion of \textit{Anosov alternative} by Dolgopyat \cite{Dolgopyat}, see also Magee, Oh, Winter \cite{MOW} for discussion and references there-in.

For an expanding Markov map $T$ satisfying UNI and equilibrium states $\mu = \mu_\phi$ associated to $T$, we will see that we can obtain a very quantitative non-concentration of derivatives using contractions of complex transfer operators in a Banach space of Lipschitz functions on $I$ associated to the potential $\phi+ib\log |T'|$ for suitable $b \in \R$ depending on the frequency $\xi$ we have when estimating $\widehat{\mu}(\xi)$ and a Lipschitz norm depending on $b$. The contraction of the transfer operator is possible because of the UNI due to \textit{Dolgopyat's method} (see Dolgopyat's works \cite{Dolgopyat1,Dolgopyat}) and the particular version we use that works for our potentials is due to Stoyanov \cite{Stoyanov} that generalises the work of Naud \cite{Naud}. Then combining the large deviation argument of \cite{JordanSahlsten} with the reduction argument \cite{bourg} to the discretised sum-product estimates of the Fourier transform under non-concentration, we can establish polynomial Fourier decay for all Gibbs measures $\mu$ for totally non-linear expanding interval maps $T$:

\begin{thm} \label{thm:nonlinear}
Suppose $T : \bigcup_{a = 1}^N I_a \to [0,1]$ is totally non-linear uniformly expanding Markov map of bounded distortions and $\mu$ is a non-atomic equilibrium state associated to a potential with exponentially vanishing variations.
\begin{itemize}
\item[(1)] If $\bigcup_{a = 1}^N I_a = [0,1]$ and $f_a$ are $C^2$, then the Fourier coefficients of $\mu$ tend to zero with a polynomial rate. 
\item[(2)] If $\{I_a : a \in \cA\}$ are disjoint and $f_a$ are analytic, then the Fourier coefficients of $\mu$ tend to zero with a polynomial rate.
\end{itemize}
\end{thm}

Comparing Theorem \ref{thm:nonlinear} to \cite{JordanSahlsten} on Gibbs measures for the Gauss map and \cite{bourg} on Patterson-Sullivan measures for convex co-compact hyperbolic surfaces, we do not get these immediately as corollaries. This is because the Gauss map has infinitely many branches and for Patterson-Sullivan measures we would need to describe the symbolic model using a subshift of finite type (SFT) coming from the Schottky structure of the Fuchsian group, where as we only consider full shift. Gauss map case would follow from Theorem \ref{thm:nonlinear}(1) if one could establish it for infinitely many branches. This should go through under a suitable tail assumption so that the large deviations are still exponentially fast and proving (or adapting) a spectral gap theorem similar to Stoyanov for countable Markov shift with a suitable tail assumption, which we will discuss a bit more later in the introduction. Furthermore, the Patterson-Sullivan case follows from Theorem \ref{thm:nonlinear}(2) if we could expand it to Markov measures, that is, when the symbolic model we use is any SFT.  Thermodynamical formalism we use and the spectral gap theorem by Stoyanov should work in this case as well. Expanding Theorem 1.1 to these settings is a good problem to do in the future.

In the language of iterated function systems, Theorem \ref{thm:nonlinear}(2) means that any Gibbs self-conformal measure in the sense of \cite{PRKS} for totally non-linear analytic IFS with self-conformal set $K$ has polynomial Fourier decay. This setting was also considered by Hochman and Shmerkin \cite{HochmanShmerkin}, where using spectral properties of the scenery flow they proved the normality of orbits of $\{n^k x\}_{k \in \N}$ for natural numbers $n \geq 2$ was established for $\mu$ almost every $x$ under total non-linearity for analytic IFSs. Theorem \ref{thm:nonlinear}, thanks to the Davenport-Erd\"os-LeVeque criterion for equidistribution \cite{DEL} gives the equidistribution for $\{s_k x\}_{k \in \N}$ at $\mu$ almost every $x$ for any strictly in increasing sequence $s_k \to \infty$. It would be interesting to see if there is connection between our result and their methods. 

When $T$ is \textit{not} totally non-linear, that is, $\tau$ is $C^1$ cohomologous to a locally constant function, then several recent works have studied this case, which we will review below:

\begin{remark}
\begin{itemize}
\item[(1)] If $\tau$ equals to a locally constant function, that is, $T$ has linear inverse branches, the $T$-invariant Cantor sets $K$ are now called \textit{self-similar sets} and the Gibbs measures $\mu$ associated to $T$ are called \textit{self-similar measures} and they are assumed to just be \textit{Bernoulli measures}: $\mu = \sum_{a \in \cA} p_a f_a(\mu)$ where $0 < p_a < 1$ and $\sum_{a \in \cA} p_a = 1$. Here the situation has also been evolving independently from similar angles and has roots in the theory of Bernoulli convolutions since the seminal works of Salem, Zygmund, Erd\"os \textit{et al.} \cite{Salem2, SZ, Erdos}.  In \cite{JialunSahlsten1} it was proved using quantitative renewal theorems for the stopping time for random walks on $\R$ that $\widehat{\mu} \to 0$ as long as the value set of $\tau$ is not contained in an arithmetic progression, i.e. the range of $\tau$ is not a lattice. These were inspired by the work on Fourier decay of the stationary measure \cite{Jialun,JialunNaudPan}. Moreover, if $\tau$ avoids a lattice in a strong way, then one can obtain logarithmic rates. Solomyak \cite{Solomyak} in particular also used Erd\"os-Kahane method to prove polynomial Fourier decay for all non-atomic $\mu$ except for a zero-Hausdorff dimensional exceptional set of parameters defining $T$. Further characterisations in the case $\tau$ is lattice was done by Br\'emont \cite{Bremont} and Varj\'u-Yu \cite{VarjuYu} using number theoretic properties of the lattice.
\item[(2)] If $\tau$ is $C^1$ cohomologous to an everywhere constant function, that is, 
$$\tau = c + g \circ T - g,$$ 
for some fixed constant $c > 0$ and the conjugacy $g : \R \to \R$ is a diffeomorphism that has non-vanishing derivative everywhere, then $\widehat{\mu}(\xi) \to 0$ with a polynomial rate for any Bernoulli measure $\mu$ associated to $T$. This is because the $g$ conjugacy can be used to find a $C^2$ diffeomorphism $h : \R \to \R$ such that the iterated function system $\{h f_a h^{-1} : a \in \cA\}$ consists of similitudes with the same contraction ratio $r_a = e^{-c} < 1$ for all $a \in \cA$ (i.e. it is \textit{homogeneous} self-similar iterated function system), $(h^{-1})'' > 0$ everywhere, and that $\mu$ is the push-forward of a self-similar measure under $h^{-1}$. Then one can then directly apply the Fourier decay of non-linear $C^2$ images of homogeneous self-similar measures proved in \cite{KaufmanSurvey} for Cantor-Lebesgue measures and \cite[Theorem 3.1]{MS} for general homogeneous self-similar measures.
\item[(3)] Going beyond the cases (1) and (2), it would be interesting to see how much this result could be relaxed or extended to Gibbs measures $\mu$ and systems $\{f_a : a \in \cA\}$ where the $C^2$ conjugacy is with an \textit{inhomogeneous} self-similar iterated function system (i.e. $\log |T'|$ is $C^1$ conjugated to a locally constant map) or for the $C^2$ conjugacy we have $h''(x) = 0$ at some $x \in \R$. The methods of \cite{MS} and \cite{KaufmanSurvey} use heavily the convolution structure of homogeneous self-similar measures so one would need to maybe rely on the renewal theoretic methods \cite{JialunSahlsten1} to address this case or Solomyak's approach for polynomial Fourier decay for almost every inhomogeneous self-similar measures using Erd\"os-Kahane method \cite{Solomyak}. There are other weaker notions of total non-linearity in the literature that could be used to generalise Theorem \ref{thm:nonlinear}. One says (e.g. \cite{Naud}) that $\tau \in C^1(I)$ is called \textit{non-lattice} if there does not exist $L : K \to m\Z$ for some $m > 0$ and $g : K \to \R$ Lipschitz such that $\tau$ satisfies the cohomological equation
$$\tau = L + g - g \circ T$$
on $K$. However, the spectral gap estimates for complex transfer operators we employ by Stoyanov \cite{Stoyanov} are not in general true. For example in Naud \cite{Naud} gave examples of non-lattice $\tau$ but not totally non-linear $\tau$ with failure of the conclusions of Naud's results \cite{Naud}, which are a special case of Stoyanov's result. We also remark of the \textit{Diophantine condition} used by Pollicott and Sharp \cite{PS}, which gives weaker contraction theorems for the transfer operators, and it could be considered as a notion between non-lattice and total non-linearity. Hence it would be interesting to see if Theorem \ref{thm:nonlinear} holds for any of these weaker notions of the non-linearity of $T$. 
\end{itemize}
\end{remark}

If we relax the condition on finite branches to infinite number of inverse branches, then we would need to rely on other works on spectral gaps of transfer operators for countable Markov maps, see for example the work of Cyr and Sarig \cite{CyrSarig}, and with suitable large deviation theory \cite{JordanSahlsten} for countable Markov maps, one should be able to extend Theorem \ref{thm:nonlinear} for infinite branches. 

Moreover, in higher dimensions it would be trickier to state the assumptions needed for Theorem \ref{thm:nonlinear} and what type of non-linear Cantor sets we could consider. Here one could for example use the symbolic setting of Avila-Gou\"ezel-Yoccoz \cite{AGY} used to characterise total non-linearity for expanding maps on Markov partitions of John domains. Here it is likely that higher dimensional analogues for Bourgain's sum-product estimates to exponential sums are needed, see for example the related work towards this goal \cite{BourgainGamburd,JialunRd,JialunNaudPan}. For the linear case, in \cite{JialunSahlsten2} Fourier decay theorems were proved for self-affine measures using non-commutativity assumption of the linear parts and these ideas were based on quantitative rates for renewal theorems. A similar strategy by proving quantitative renewal theorems is likely to be possible in the commuting case under dense rotations. For example, in the case of self-similar measures under a spectral gap assumption for the random walks on $SO(d)$, for $d \geq 3$, was done by Lindenstrauss and Varj\'u \cite{LV}.  However, establishing rates for renewal theorems on $SO(d)$ is a difficult problem and relies on proving a spectral gap on the random walks on $SO(d)$ under a dense rotations condition, which is only known in the case of algebraic entries by Benoist and Saxc\'e \cite{BenoistSaxce} recently. 

Finally, we conclude with a problem posed to us by S. Dyatlov and P. Shmerkin: 

\begin{question}\label{question} Is the polynomial decay rate in Theorem \ref{thm:nonlinear} is \textit{continuous} in the defining parameters? For example, if a sequence of totally non-linear systems $(T_n)_{n \in \N}$ approaches pointwise the doubling map $x \mapsto 2x \mod 1$ as $n \to \infty$ and we consider the Gibbs measures $\mu_n$ of maximal Hausdorff dimension for $T_n$, so by Theorem \ref{thm:nonlinear}, we have $|\widehat{\mu}(\xi)| = O(|\xi|^{-\alpha_n})$ as $|\xi| \to \infty$, then does $\alpha_n \to 0$ as $n \to \infty$? 
\end{question}

Here it is likely one would need to study the continuity of the spectral gap \cite{Stoyanov} of the transfer operators under perturbations we used for non-concentration estimates, entropy, Lyapunov exponents and the pressure in the large deviation bounds \cite{JordanSahlsten} of $\mu_n$ going back to field of  \textit{linear response} \cite{Ruelle} and the discretised sum-product estimates \cite{Bourgain2010} we use.\\

\textbf{Strategy and organisation of the proof.} We will begin the proof in Section \ref{sec:regularisation} where we first introduce the required thermodynamic notation and symbolic structure following similar notations as in \cite{JordanSahlsten} on Fourier transforms of Gibbs measures for the Gauss map and \cite{bourg} on Patterson-Sullivan measures. Here we will state large deviation results for the distortion function $\tau = \log |T'|$ and Hausdorff dimension for Gibbs measures $\mu$, which allow us to extract a regular part of the measure $\mu$ and decompose quantitatively the Fourier transform $\widehat{\mu}(\xi)$ into regular and irregular components.

For the extracted regular part of $\mu$, in Section \ref{sec:reductionsexponentialsums} we first begin as Bourgain and Dyatlov \cite{bourg} do by adapting bounds of bounded distortions of $T$ to reduce $|\widehat{\mu}(\xi)|^2$ into an exponential sum arising from the distortion function $\tau = \log |T'|$. Then in Section \ref{sec:additive} we see that the exponential sum estimates can be bounded using a consequence of the discretised sum-product theorem established by Bourgain in \cite{Bourgain2010}, and we prove a particular form adapted to our setting that takes into account the multiplicative errors arising from large deviations for $\tau$. However, the assumptions to apply the sum-product estimates require a non-concentration property for the derivatives of $T$, which is possible thanks to total non-linearity of $T$ and an argument using the contraction for the complex transfer operators associated to the potential $\phi - s\tau$ for suitable $s \in \C$ depending on $\xi$ established by Stoyanov \cite{Stoyanov}.

Finally in Section \ref{sec:completion} we complete the proof by carefully choosing the right parameters so that the large deviation estimates, non-concentration estimates and decay theorems for exponential sums are satisfied.

\section{Regular part of the measure and large deviations} \label{sec:regularisation}

\subsection{Symbolic and thermodynamic preliminaries} \label{sec:thermo} Let us collect here all the symbolic notations we will use throughout the proof. Write $\cA = \{1,\dots,N\}$ and let $\cA^*$ the collection of all finite words with alphabet $\cA$. 

\begin{itemize}
\item[(1)] For $\a = (a_1,\dots,a_n) \in \cA^n$, we write $f_\a = T_{a_1}^{-1} \circ \dots \circ T_{a_n}^{-1}$.
\item[(2)] Given
$$ {\bf A }= ({\a}_0,\a_1,...,{\a}_k) \in  (\cA^n)^{k+1} \text{   and   } {\b} = ({\b}_1,...,{\b}_k) \in  (\cA^n)^k $$
define the following concatenation operators:
$${\A}*{\B}:={\a}_0{\b}_1{\a}_1{\b}_2...{\a}_{k-1}{\b}_k{\a}_k \quad \text{and} \quad {\A}\#{\B}:={\a}_0{\b}_1{\a}_1{\b}_2...{\a}_{k-1}{\b}_k.$$
\item[(3)] Write also $I_\a = f_\a([0,1])$, $\a \in \cA^n$, $n \in \N$.
\end{itemize}

Let us now introduce some thermodynamical formalism. Here we refer to \cite{Keller,Sarig} for background on thermodynamic formalism and \cite{Mattila95} on geometric measure theory and Hausdorff dimension. For the rest of this paper, we fix a continuous $\phi : I  = \bigcup_{a = 1}^N I_a\to \R$, $\phi < 0$. The \textit{variation} of $\phi$ at the generation $n$ is defined by
$$\mathrm{var}_n(\phi) := \sup_{\a \in \cA^{n}} \sup\{|\phi(u) - \phi(v)| : u,v \in I_{\a}\}.$$
In this paper we will assume $\phi$ is regular in the sense that $\mathrm{var}_n(\phi) \to 0$ exponentially as $n \to \infty$. We define the \textit{transfer operator} associated to $\phi$ as
$$\cL_\phi g(x) = \sum_{y : T(y) = x} e^{\phi(y)} g(y)$$
for continuous $g : I \to \C$. Note that we consider $T$ which is conjugate to the full shift on the word space $\mathcal{A}^*$, so the transfer operator is defined as such for any $x \in I$. The dual operator $\cL_\phi^*$ acting on the space of measures on $K$ is then defined by the formula
$$\int_K g \, d\cL_\phi^* \mu := \int_K \cL_\phi g \, d\mu.$$
Here we will consider the unique probability measure $\mu = \mu_\phi$ on $K$ satisfying $\cL_\phi^* \mu = \mu$ maximising the pressure formula
$$P(\phi) = \sup\Big\{h_\mu(T) + \int_K \phi \, d\mu : \mu = T\mu \Big\}$$
where $h_\mu(T)$ is the entropy of $\mu$ with respect to $\mu$, with pressure $P(\phi) = 0$ and satisfying the \textit{Gibbs condition}
\begin{align}\label{eq:Gibbs} C^{-1} e^{S_n \phi(f_\a(x)) } \leq \mu(I_\a) \leq C e^{S_n \phi(f_\a(x)) }\end{align}
for all $x \in [0,1]$ and for some $C > 0$, where $S_n \phi(x) := \phi(x) + \phi(Tx) + \dots + \phi(T^{n-1} x)$ is the \textit{Birkhoff sum} of $\phi$. For brevity, we will from now on define the \textit{Birkhoff weights} at $x \in [0,1]$ by
$$w_{\a}(x) := e^{S_n \phi(f_\a(x)) }, \quad \a \in \cA^n$$
so by \eqref{eq:Gibbs} we have $C^{-1} w_\a(x) \leq \mu(I_\a) \leq C w_\a(x)$ for all $x \in I$. We note that it is known, see for example \cite{JordanSahlsten} and the references there-in that one can relax the zero pressure condition but still have the same properties as we claim here.

\subsection{Large deviation estimates for Gibbs measures}

We need to find a large regular part of the measure $\mu$ in terms of the Lyapunov exponent and Hausdorff dimension, which allow us to manipulate the Fourier transforms. In Bourgain and Dyatlov \cite{bourg} they dealt with Patterson-Sullivan measures which automatically are Ahlfors-David regular, which is stronger than the Gibbs condition. Large deviations allow us to extract a ``large part'' of the support with similar Ahlfors-David regular behaviour for $\mu$. Here is also where we need the finite Lyapunov exponent for $\mu$.

\begin{theorem}[Large deviations]\label{thm:largedev}
Let $\mu$ be the equilibrium state associated to $\phi$ with Lyapunov exponent $\lambda = \int \tau \, d\mu > 0$ and Hausdorff dimension $\delta = \Hd \mu > 0$. Then we have that for any $\varepsilon >0$, there exists $\delta_0(\eps) >0$ and $n_1(\varepsilon) \in \N$  such that
		$$ \mu \Big(\Big\{x \in [0,1] : \Big| \frac{1}{n} S_n \tau (x) - \lambda \Big| \geq \varepsilon \quad  \mathrm{ or } \quad  \Big| \frac{S_n \phi (x)}{-S_n \tau (x)} - \delta \Big| \geq \varepsilon \Big\} \Big) \leq e^{-\delta_0(\eps) n} $$	
		 for all $n\geq n_1(\eps)$.
\end{theorem}

The proof of the above in this form was given in \cite[Corollary 4]{JordanSahlsten} for the Gauss map $x \mapsto 1/x \mod 1$ on countable alphabets, where a tail assumption for $\mu$ is imposed in terms of the behaviour of $\tau$ at the tail. However, the proof works for any Markov map with the condition (1), (2) and (3) we assume in the introduction and the equilibrium states $\mu = \mu_\phi$ associated to $T$ with exponentially vanishing variations. In the finite alphabets we also do not have a tail, so we can apply \cite[Corollary 4]{JordanSahlsten} in this form.

\subsection{Regular words $\cR_n(\eps)$ and regular blocks $\cR_n^k(\eps)$}

Let us now use the large deviations to construct regular words and blocks of words that we will use in our analysis of the Fourier transform of $\mu$. Fix now $\eps > 0$ and $n \in \N$. Write
$$A_n(\eps) := \Big\{x \in [0,1] : \Big| \frac{1}{n} S_n \tau (x) - \lambda \Big| < \varepsilon \quad  \mathrm{ and } \quad  \Big| \frac{S_n \phi (x)}{-S_n \tau (x)} - \delta \Big| < \varepsilon \Big\}$$

\begin{definition}[Regular words and blocks]
Fix $\eps > 0$ and $\eps_0 > 0$.
\begin{itemize}
\item[(1)] For a generation $n \in \N$ the set of \textit{regular words} of length $n$:
$$\cR_n(\eps) = \cR_n(\eps,\eps_0) := \bigcap\limits_{j=\lfloor \eps_0 n \rfloor}^n \{ {\a} \in \cA^n : I_{{\a|_j}} \subset A_j(\varepsilon) \}$$
Note that unlike \cite{JordanSahlsten}, we will require $\lfloor \eps_0 n \rfloor$-regularity as opposed to $\lfloor n/2 \rfloor$. We note that the $\eps_0$ is the exponent $\eps_0 > 0$ for the non-concentration for distortions we will eventually find in Lemma \ref{thm:nonconc} using total non-linearity of $T$. It does not depend on $\eps$ so we will suppress it from the notations until we will need it.
\item[(2)] For a generation $n \in \N$ and parameter $k \in \N$, define a \textit{regular block} $\A = \a_1 \dots \a_k$ of length $k$ to be the concatenation of $k$ regular words $\a_j \in \cR_n(\eps)$, $j = 1,\dots,k$, of length $n$. We denote the set of such regular blocks by $\mathcal{R}_n^{k}(\eps) = \cR_n^k(\eps,\eps_0)$.
	\end{itemize}
\end{definition}
For $k \in \N$, we shall consider the corresponding geometric points to be 
	$$R_n^k(\eps) := \bigcup_{\A \in \mathcal{R}_n^{k}(\eps)}I_{\A} \subset [0,1] \quad \text{and} \quad R_n(\eps) := R_n^1(\eps)$$
In the next lemma we will define the basic properties of the regular words and blocks, which we will use throughout the proof. In particular this allows us to control the lengths and $\mu$ measures of the intervals associated to the words and blocks. Thanks to the large deviation principle we know ``most'' words in terms of $\mu$ mass satisfy these properties.
	
\begin{lemma} \label{lma:regularmarkov}
	For any $j \in \N$ define the exponential number 
	$$C_{\varepsilon,j}:= e^{\varepsilon j}$$ 
	and assume that $n$ is chosen large enough so that
	$$\frac{\log 4}{\eps_0 n}<\varepsilon/2, \,\,\,\,\, \frac{\log 4C^2}{\log(\gamma^{2 \eps_0n})}<\varepsilon /2 \text{\,\,\,\,\,\and\,\,\,\,\,} \frac{e^{-\eps_0 \eps_0 n}}{1-e^{-\eps_0}} < e^{-\eps_0 \eps_0 n/2},$$
	where $\gamma > 1$ satisfies $|(T^n)'(z)| \geq C \gamma^n$ for all $n \in \N$ and $z \in I$.
	For some $n$-regular word ${\a}\in \mathcal{R}_n(\eps)$ and $j \in \{ \lfloor \eps_0 n \rfloor,..., n \}$ we have that the following hold:
	\begin{itemize}
		\item[(i)]
		the size of the derivative $|f'_{{\a|_j}}|$ satisfies
		$$ \frac{1}{16} C_{\varepsilon,j}^{-1}e^{-\lambda j} \leq |f'_{{\a|_j}}| \leq C_{\varepsilon,j}e^{-\lambda j}$$
		and hence so does the length $|I_{{\a|_j}}|$;
		\item[(ii)] The measure satisfies
		$$ C^{-1}\cdot C_{\varepsilon,j}^{-3\lambda}e^{-s\lambda j} \leq \mu(I_{{\a|_j}}) \leq C\cdot C_{\varepsilon,j}^{3\lambda}e^{-s\lambda j};$$
		\item[(iii)] For all $x \in [0,1]$, the Birkhoff weights $w_{\a|_j}(x) = e^{S_j \phi(f_{\a|j}(x)) }$ satisfy
		$$ C_{\varepsilon,j}^{-3\lambda}e^{-s\lambda j} \leq w_{\a|_j}(x) \leq C_{\varepsilon,j}^{3\lambda}e^{-s\lambda j}. $$
		\item[(iv)] The cardinality
		$$\tfrac{1}{2}C^{-1 }C_{\eps,n}^{-3\lambda} e^{\lambda s n} \leq \sharp \cR_n(\eps) \leq C C_{\eps,n}^{3\lambda} e^{\lambda s n}$$
	\end{itemize}
Then for $k \in \N$ we have that if $n \to \infty$,
	$$ \mu ([0,1] \setminus R_n^k(\eps)) = O(e^{-\delta_0(\eps/2) \eps_0 n/2}) $$
	where $\delta_0 (\varepsilon /2)$ is given to us Theorem \ref{thm:largedev}.
\end{lemma}

\begin{proof} Parts (i), (ii), and (iii) are done in \cite[Lemma 5]{JordanSahlsten} and the part (iv) follows from the bounds for $\mu(I_\a)$ and combining with the measure bound for $\mu ([0,1] \setminus R_n(\eps))$. 

For the upper bound of $\mu ([0,1] \setminus R_n^k(\eps))$, it is sufficient to prove that
	$$ \bigcap\limits_{i=0}^{k-1} (T^{-1})^{ni} \Big( \bigcap\limits_{j=\lfloor \eps_0 n \rfloor}^n A_j(\varepsilon /2) \Big) \subset R_n^k(\eps) $$
	since we have that
	\begin{align*}
	\mu([0,1] \setminus R_n^k(\eps)) &\leq \mu \Big([0,1] \setminus \bigcap\limits_{i=0}^{k-1} (T^{-1})^{ni} \Big( \bigcap\limits_{j=\lfloor \eps_0 n \rfloor}^n A_j(\varepsilon /2) \Big) \Big)\Big)\\
	& \leq \sum\limits_{i=0}^{k-1} \mu \Big( [0,1] \setminus  (T^{-1})^{ni} \Big( \bigcap\limits_{j=\lfloor \eps_0 n \rfloor}^n A_j(\varepsilon /2) \Big)\Big) \\
	&\leq k\mu \Big( [0,1] \setminus \Big( \bigcap\limits_{j=\lfloor \eps_0 n \rfloor}^n A_j(\varepsilon /2) \Big) \Big) 
	\\
	&\leq ke^{-\delta_0(\eps/2) \eps_0 n/2}
	\end{align*}
	where the details of the last inequality are given in \cite[Lemma 6]{JordanSahlsten}. 
	
	We now prove the claim. Let $\B \in (\cA^n)^k$ be a word such that $T^{ni}f_{\B}x \in A_j(\varepsilon /2)$ for all $i = 0,1,...,k-1$ and all $j=\lfloor \eps_0 n \rfloor,...,n$. We want to prove that $f_{\B}x \in R_n^k(\eps)$. By definition of $R_n^k(\eps)$, it is enough for us to prove that $f_{\B}x \in I_{\A}$ for some $\A \in \mathcal{R}_n^{k}(\eps)$. So we can just prove that $\B \in \mathcal{R}_n^{k}(\eps)$. By definition of $\mathcal{R}_n^{k}(\eps)$, we need to prove that $I_{(\sigma^n)^i \B|_j} \subset A_j (\varepsilon)$ for all $i= 0,1,...,k-1$ and $j=\lfloor \eps_0 n \rfloor,...,n$. If we have $y \in [0,1]\setminus K$, then $f_{(\sigma^n)^i\B|_j}y$ is a general point in $I_{(\sigma^n)^i\B|_j}$ (we may equivalently consider the point $T^{ni}f_{\B|_j}y$). So we want to prove that $f_{(\sigma^n)^i \B|_j}y \in A_j(\varepsilon)$. Using the assumptions on $\B$ we have that
	\begin{align*}  \Big| \frac{1}{j}S_j\tau (f_{(\sigma^n)^i\B|_j}y)-\lambda \Big| &\leq \Big| \frac{1}{j}S_j\tau (f_{(\sigma^n)^i\B|_j}y) - \frac{1}{j}S_j\tau (f_{(\sigma^n)^i\B}x) \Big| + \frac{\varepsilon}{2} \\
	& = \frac{\varepsilon}{2} + \frac{1}{j}\log \frac{|f'_{(\sigma^n)^i\B|_j}(f_{(\sigma^n)^{i+1}\B}y)|}{|f'_{(\sigma^n)^i\B|_j}(f_{(\sigma^n)^{i+1}\B}x)|}\\
	&\leq \frac{\varepsilon}{2}+\frac{\log 4}{j}\\
	&\leq \varepsilon
	\end{align*}
	by choice of $n$. Now for the second condition we see that
	\begin{align*}
	 \Big| \frac{S_j\phi(f_{(\sigma^n)^i\B|_j}y)}{-S_j\tau(f_{(\sigma^n)^i\B|_j}y)}-\delta \Big| & \leq \Big| \frac{S_j\phi(f_{(\sigma^n)^i\B|_j}y)}{-S_j\tau(f_{(\sigma^n)^i\B|_j}y)} - \frac{S_j\phi(f_{(\sigma^n)^i\B}x)}{-S_j\tau(f_{(\sigma^n)^i\B}x)} \Big| + \frac{\varepsilon}{2}\\
	 & \leq \frac{\log 4C^2}{\log (c\gamma^{2k})}+ \frac{\varepsilon}{2} \\
	 &< \varepsilon
	 \end{align*}
	 as in the proof of \cite[Lemma 6]{JordanSahlsten}.
\end{proof}

\section{Reduction to exponential sums}\label{sec:reductionsexponentialsums}

For a regular word ${\a}$, we define $x_{\a}\in I_{\a}$ to be the center point of this construction interval, recall $I_\a = f_\a ([0,1])$, where $T : I \to \R$ is the expanding map. For a block ${\bf A }= ({\a}_0,,...,{\a}_k) \in \cR_n^{k+1}(\eps)$ and an index $j \in \{1,\dots,k\}$, we define the map $\zeta_{j,\A} : \cR_n(\eps) \to \R$ by
$$ \zeta_{j,{\A}}({\b}) := e^{2\lambda n}f_{{\a_{j-1}\b}}'(x_{{\a_j}}), \quad {\b} \in \mathcal{R}_n(\eps).$$
Given $\xi \in \R$, we will now reduce $|\widehat{\mu}(\xi)|^2$ into an estimate consisting of exponential sums consisting of products of the combinatorial functions $ \zeta_{j,{\A}}({\b})$ involving derivatives of $T$.

Below will write $|\xi| \sim N$ means that there exists a constant $c > 0$ such that $c^{-1} N \leq |\xi| \leq cN$ and that $A \lesssim_\mu B$ means that $A \leq C_\mu B$ for some constant $C_\mu > 0$ depending on $\mu$. 

\begin{prop}\label{lma:exponentialsum}  Fix $\eps > 0$,  $\xi \in \R$, $\eps_0 > 0$, $k \in \N$ and $n \in \N$ and write
	$$J_n(\eps) = J_n(\eps,\eps_0) := \{ \eta \in \mathbb{R}:e^{\eps_0 n/2} \leq |\eta| \leq C_{\varepsilon,n} e^{\eps_0 n}  \},$$
	where recall $C_{\eps,n} = e^{\eps n}$. Then whenever $|\xi| \sim e^{(2k+1) n \lambda} e^{\eps_0 n}$ we can bound
	\begin{align*} | \widehat{\mu}(\xi)| ^2 \lesssim_\mu \,\,&C_{\varepsilon,n}^{(2k+1)\lambda}e^{-\lambda (2k+1)\delta n} \sum\limits_{{\A} \in \cR_n^{k+1}(\eps)} \sup\limits_{\eta \in J_n(\eps)} \Big| \sum\limits_{{\B \in \cR_n^k(\eps)}} e^{2\pi i \eta \zeta_{1,{\A}}({\b_1})...\zeta_{k,{\A}}({\b}_k)} \Big|  \\
	& + e^{2k}C_{\varepsilon,n}^{k+2} e^{-\lambda n}e^{\eps_0 n} +\mu([0,1]\setminus R_n^{k+1}(\eps))^2+ C_{\varepsilon,n}^{4k\lambda} \rho^{2n} + \mu(R_n(\eps)^c)+C_{\varepsilon,n}^2e^{- \delta \eps_0 n/2}.
	\end{align*}
Here $0 < \rho < 1$ comes from the variations of $\phi$: $\mathrm{var}_n(\phi) = O(\rho^n)$ as $n \to \infty$.
\end{prop}

\begin{remark}
\begin{itemize}
\item[(1)] Let us mention a few words about the quantities $\eps_0 > 0$ and $k \in \N$ appearing in Proposition  \ref{lma:exponentialsum}. These will be chosen later and they come from the parameters in our sum-product bounds. The value $\eps_0 > 0$, which is coming from non-concentration of derivatives eventually, is fixed in Remark \ref{rmk:parameters}(2) and the number $k \in \N$ coming from the sum-product estimates and it depends on $\eps_0$ is fixed in Lemma \ref{lma:newexponential}. We will use these to choose the generation $n$ using $\xi$, where we want to form our regular blocks $\cR_{n}^k(\eps)$ and the decomposition of the Fourier transform $\widehat{\mu}(\xi)$ using exponential sums.
\item[(2)] Proposition \ref{lma:exponentialsum} in some sense captures all the quantities that could influence the Fourier decay, and in particular here if one would like to study the \textit{continuity} of the parameters (recall Question \ref{question}), one would need to make this estimate more quantitative. In particular we will see that the main exponential sum term will be controlled by the sum-product estimates that gives the $k \in \N$, which we can use under suitable non-concentration condition that depends on $\eps_0 > 0$, which we will later see coming from total non-linearity using a spectral gap argument. Then the rest of the terms depend on the variations decay rate $0 < \rho < 1$ of the potential $\phi$: $\mathrm{var}_n(\phi) = O(\rho^{n})$, as $n \to \infty$, large deviation bounds rates $\mu([0,1]\setminus R_n^{k+1}(\eps))$ and $\mu(I\setminus R_n(\eps))$, which depend on the shape of the pressure function defined by $\phi$ and $\tau$, and the Hausforff dimension $\delta = \Hd \mu$ and Lyapunov exponent $\lambda = \int \log |T'| \, d\mu$ of $\mu$ with respect to $T$.
\end{itemize}
\end{remark}

Let us now proceed to prove Proposition \ref{lma:exponentialsum}. First we will first use the $\cL_\phi^*$ invariance of $\mu$ to obtain the following estimate. Here, recall the concatenation notation
$${\A}*{\B} ={\a}_0{\b}_1{\a}_1{\b}_2...{\a}_{k-1}{\b}_k{\a}_k \in (\cA^{n})^{2k+1}$$
for the blocks
$$ {\bf A }= ({\a}_0,...,{\a}_k) \in  (\cA^n)^{k+1} \text{   and   } {\b} = ({\b}_1,...,{\b}_k) \in  (\cA^n)^k.  $$

\begin{lemma} \label{lma:FourierReg} For all $\eps > 0$, $\xi \in \R$, $n \in \N$ and $k \in \N$, we have
	$$|\widehat{\mu}(\xi)|^2 \lesssim_\mu \Big|\sum\limits_{\substack{\A \in \mathcal{R}_n^{k+1}(\eps)}} \sum_{ \B \in \mathcal{R}_n^{k}(\eps)} \int e^{-2\pi i \xi f_{\A \ast \B}(x)} w_{\A \ast\B}(x) \, d\mu(x)\Big|^2 + \mu([0,1]\setminus R_n^{k+1}(\eps))^2,$$
	where $w_{\A \ast \B}(x) = e^{S_{(2k+1)n} \phi (f_{\A \ast \B}(x))}$ for $\A \in  (\cA^n)^{k+1}$ and $\B \in (\cA^n)^k$.
\end{lemma}
\begin{proof}
Given $\xi \in \R$, write
$$h(x) := e^{-2\pi i \xi x}, \quad x \in \R.$$ 
Since $\mu = \cL_\phi^* \mu$, we have
	$$\widehat{\mu}(\xi) = \int h(x) \, d\mu(x) = \int \cL_\phi^{(2k+1)n} h(x) \, d\mu(x) = \int (\cL_\phi^{n})^{2k+1} h(x) \, d\mu(x).$$
	By the definition of the transfer operator
$$ (\cL_\phi^n)^{2k+1} h(x) = \sum\limits_{\mathbf{C} \in (\cA^n)^{2k+1}} w_{\mathbf{C}}(x)h(f_{\mathbf{C}}x) =  \sum\limits_{\substack{\A \in (\cA^n)^{k+1} \\ \B \in (\cA^n)^k}}w_{\A * \B}(x)h(f_{\A * \B}x).$$
	This splits using $\cR_n^k$ and $(\cA^n)^k \setminus \cR_n^k$ to
	$$ \sum\limits_{\substack{\A \in \mathcal{R}_n^{k+1}(\eps), \\ \B \in \mathcal{R}_n^{k}(\eps)}}w_{\A * \B}(x)h(f_{\A * \B}x)  + \sum\limits_{\substack{\A \in (\cA^n)^{k+1}\setminus \mathcal{R}_n^{k+1}(\eps) \\ \text{or }\B \in (\cA^n)^k\setminus \mathcal{R}_n^{k}(\eps)}}w_{\A * \B}(x)h(f_{\A * \B}x) . $$
	Integrating over $x$, we can bound the modulus of the right-hand side by
	\begin{align*}
	\Big| \int \sum\limits_{\substack{\A \in (\cA^n)^{k+1}\setminus \mathcal{R}_n^{k+1}(\eps) \\ \text{or }\B \in (\cA^n)^k\setminus \mathcal{R}_n^{k}(\eps)}}w_{\A * \B}(x)h(f_{\A * \B}x) \, d\mu \Big| & \leq \int \sum\limits_{\substack{\A \in (\cA^n)^{k+1}\setminus \mathcal{R}_n^{k+1}(\eps) \\ \text{or }\B \in (\cA^n)^k\setminus \mathcal{R}_n^{k}(\eps)}}w_{\A * \B}(x) \,d\mu \\
	& \lesssim_\mu \sum\limits_{\substack{\A \in (\cA^n)^{k+1}\setminus \mathcal{R}_n^{k+1}(\eps) \\ \text{or }\B \in (\cA^n)^k\setminus \mathcal{R}_n^{k}(\eps)}} \mu (I_{\A * \B}) \\
	& \leq \mu ([0,1]\setminus R_n^{k+1}(\eps) )+\mu([0,1]\setminus R_n^k(\eps)).
	\end{align*}
	We get the required result by noting that $R_n^{k+1}(\eps) \subset R_n^k(\eps)$, which follows by the fact that for any $\A \in \mathcal{R}_n^{k+1}(\eps)$ we have that there exists $\B \in \mathcal{R}_n^{k}(\eps)$ such that $\A = \B \a_k$ for some $\a_k \in \cR_n(\eps)$. Conclude using $|a+b|^2 \leq 2|a|^2+2|b|^2$ for complex numbers.
\end{proof}

Next, in the sums obtained in right-hand side of the estimate Lemma \ref{lma:FourierReg}, we will want to transfer the integral into exponential sums. This will be possible by the decaying variations of the potential $\phi$ defining the Gibbs measure and the bounded distortion assumption on $T$ from the assumptions in Theorem \ref{thm:nonlinear}. We have the following quantitative estimate:

\begin{lemma}\label{lma:removeintegral}
	 For all $\xi \in \R$, $n \in \N$ and $k \in \N$, we have
	\begin{align*}&\Big|\sum\limits_{\substack{\A \in \mathcal{R}_n^{k+1}(\eps)}} \sum_{ \B \in \mathcal{R}_n^{k}(\eps)} \int e^{-2\pi i \xi f_{\A \ast \B}(x)} w_{\A \ast\B}(x) \, d\mu(x)\Big|^2\\
	& \lesssim_\mu C_{\varepsilon,n}^{(2k-1)\lambda}e^{-(2k-1)\lambda \delta n}\sum\limits_{\substack{\A \in \mathcal{R}_n^{k+1}(\eps)}} \sum_{ \B \in \mathcal{R}_n^{k}(\eps)} \Big|\int e^{-2\pi i\xi f_{\A*\B}(x)}w_{\a_k}(x) d\mu(x) \Big|^2+ C_{\varepsilon,n}^{4k\lambda} \rho^{2n}.
	\end{align*}
\end{lemma}

\begin{proof}
	Since $\phi$ is locally H\"older, we know that there exists a constant $C > 0$ and $0 < \rho < 1$ such that for any $m \in \N$ we have
	$$\sup_{\w \in \N^{m}} \sup\{|\phi(u) - \phi(v)| : u,v \in I_{\w}\} \leq C\rho^m.$$
	Choose a point $y \in [0,1]$ such that $x_{\a_k} = f_{\a_k}(y)$. Then we have that
	$$\frac{w_{\A \# \B}(f_{\a_k}x)}{w_{\A \# \B}(x_{\a_k})} = \exp(S_{2kn}\phi(f_{\A*\B}(x)) - S_{2kn}\phi(f_{\A*\B}(y)) ),$$
	where we used the notation
	$$ {\A}\#{\B}:={\a}_0{\b}_1{\a}_1{\b}_2...{\a}_{k-1}{\b}_k \in (\cA^n)^{2k}$$
	for 
	$$ {\bf A }= ({\a}_0,...,{\a}_k) \in  (\cA^n)^{k+1} \text{   and   } {\b} = ({\b}_1,...,{\b}_k) \in  (\cA^n)^k.  $$
	This gives using $|\A*\B| = (2k + 1)n$ that
	$$|S_{2kn}\phi(f_{\A*\B}(x)) - S_{2kn}\phi(f_{\A*\B}(y))| \leq \sum_{j = 0}^{2kn-1} C\rho^{2kn + n - j} \leq \frac{C}{1-\rho} \rho^{n+1} =: C_0 \rho^{n+1}.$$
	Hence
	$$\exp(-C_0 \rho^n) \leq \frac{w_{\A \# \B}(f_{\a_k}x)}{w_{\A \# \B}(x_{\a_k})} \leq \exp(C_0 \rho^n).$$
	Rearranging this result we have that
	$$ |w_{\A\#\B}(f_{\a_k}x)-w_{\A\#\B}(x_{\a_k})| \leq \max\{ |\exp(\pm C_0\rho^n)-1| \} w_{\A\#\B}(x_{\a_k}). $$
	Hence since $|e^{i\theta}| = 1$ we have that
	$$ |e^{-2\pi i \xi f_{\A*\B}(x)}w_{\A*\B}(x)-w_{\A\#\B}(x_{\a_k})e^{-2\pi i \xi f_{\A*\B}(x)}w_{\a_k}(x)| \leq C^{2k}C_{\varepsilon,n}^{2k\lambda}e^{-(2k+1)\lambda \delta n} \cdot C_0 \rho^n$$
	where we use that fact that $w_{\A*\B}(x)=w_{\A \# \B}(f_{\a_k}(x))w_{\a_k}(x)$ and that
	$$w_{\A\#\B}(x_{\a_k}) \leq C^{2k}C_{\varepsilon,n}^{2k\lambda}e^{-2k\lambda \delta n}$$
	by Lemma \ref{lma:regularmarkov}.
	
	Thus when summing over $\A \in \mathcal{R}_n^{k+1}(\eps)$ and $\B \in \mathcal{R}_n^{k}(\eps)$ below:
	\begin{align*}
	& \Big|\sum_{\bf A,B} \int e^{-2\pi i \xi f_{\A \ast \B}(x)} w_{\A \ast\B}(x) \, d\mu(x) - \sum\limits_{\A,\B} w_{\A \# \B}(x_{\a_k}) \int e^{-2\pi i \xi f_{\A * \B}(x)}w_{\a_k}(x) \, d\mu(x) \Big|\\
	& \leq \sum\limits_{\A,\B} \int |e^{-2\pi i \xi f_{\A*\B}(x)}w_{\A*\B}(x)-w_{\A\#\B}(x_{\a_k})e^{-2\pi i \xi f_{\A*\B}(x)}w_{\a_k}(x)| \, d\mu(x)\\
	& \leq \sum\limits_{\A,\B} C^{2k}C_{\varepsilon,n}^{2k\lambda}e^{-(2k+1)\lambda \delta n} \cdot C_0 \rho^n \\
	& \leq C_1C_{\varepsilon,n}^{2k\lambda} \rho^n
	\end{align*}
	where we use the fact that we have an upper bound on the number of block combinations $\A \in \mathcal{R}_n^{k+1}(\eps)$ and $\B \in \mathcal{R}_n^{k}(\eps)$, which is given by $C^{2k+1}C_{\varepsilon,n}^{(2k+1)\lambda}e^{(2k+1)\lambda \delta n}$ by Lemma \ref{lma:regularmarkov}. 
	
	Moreover, using the Cauchy-Schwarz inequality we get that 
	\begin{align*}
	& \Big| \sum\limits_{\A,\B} w_{\A \# \B}(x_{\a_k}) \int e^{-2\pi i \xi f_{\A * \B}(x)}w_{\a_k}(x) \, d\mu(x)\Big|^2\\
	& \leq C^{2k-1}C_{\varepsilon,n}^{(2k-1)\lambda}e^{-(2k-1)\lambda \delta n}\sum\limits_{\A,\B} \Big|\int e^{-2\pi i \xi f_{\A*\B}(x) }w_{\a_k}(x)\, d\mu(x) \Big|^2
	\end{align*}
	Using $|a+b|^2\leq2|a|^2+2|b|^2$ for $a,b \in \mathbb{C}$, we get the result.
\end{proof}

Now we are ready to finish the proof of Lemma \ref{lma:exponentialsum}:

\begin{proof}[Proof of Lemma \ref{lma:exponentialsum}]
	Using Lemmas \ref{lma:FourierReg} and \ref{lma:removeintegral}, we obtain:
	$$ | \widehat{\mu}(\xi)| ^2 \lesssim_\mu e^{-\lambda (2k-1)\delta n} \sum\limits_{{\bf A,B}} \Big| \int e^{ -2\pi i \xi f_{{\bf A*B}}(x)} w_{{\a}_k}(x) \, d\mu(x) \Big|^2+\mu([0,1]\setminus R_n^{k+1}(\eps))^2+ e^{-\eps_2 n/2}$$
	with the sum over $\A \in \mathcal{R}_n^{k+1}(\eps)$ and $\B \in \mathcal{R}_n^{k}(\eps)$. Consider the term
	$$C_{\varepsilon,n}^{(2k-1)\lambda}e^{-\lambda (2k-1)\delta n} \sum\limits_{{\bf A,B}} \Big| \int e^{ -2\pi i \xi f_{{\bf A*B}}(x)} w_{{\a}_k}(x) \, d\mu(x) \Big|^2,$$
	which, when opening up, is equal to
	$$C_{\varepsilon,n}^{(2k-1)\lambda}e^{-\lambda (2k-1)\delta n} \sum_\A \iint  w_{\a_k}(x) w_{\a_k}(y) \sum\limits_{{\B}}  e^{2\pi i \xi (f_{{\bf A*B}}(x) - f_{{\bf A*B}}(y))} \, d\mu(x) \, d\mu(y).$$
	Taking absolute values, and using the bound for $w_{\a_k}(x)w_{\a_k}(y) \lesssim C_{\eps,n}^2 e^{-2\lambda \delta n}$, this is bounded from above by
	$$\lesssim C_{\varepsilon,n}^{(2k+1)\lambda} e^{-\lambda (2k+1)\delta n} \sum_\A \int\int \Big|\sum\limits_{{\B}}  e^{2\pi i \xi (f_{{\bf A*B}}(x) - f_{{\bf A*B}}(y))}  \Big| \, d\mu(x) \, d\mu(y).$$

	Consider a fixed block $\A \in \cR_n^{k+1}(\eps)$. Given $x,y \in I$, define $\hat{x}:=f_{\a_k}(x)$ and $\hat{y}:=f_{\a_k}(y)$ both of which are in $I_{\a_k}$. We also have that $f_{\A*\B}(x)=f_{\A\#\B}(\hat{x})$ and $f_{\A*\B}(y)=f_{\A \# \B}(\hat{y})$. By the Fundamental Theorem of Calculus we have that
	$$ f_{\A*\B}(y)-f_{\A*\B}(x)= \int_{\hat{x}}^{\hat{y}} f_{\A \# \B}'(t) dt. $$
	By applying the chain rule $k$ times, we have that there exists $t_i \in I_{\a_i}$ for $i=1,...,k$ such that
	$$ f_{\A\#\B}'(t) = f_{\a_0\b_1}'(t_1)f_{\a_1\b_2}'(t_2)...f_{\a_{k-1}\b_k}'(t_k) $$
	where $t_k=t$. 
	
	Let us now invoke the bounded distortion assumption from $T$, recall the assumption (3) of $T$ from the introduction. It implies that there exists $B > 0$ such that for all ${\a} \in \cA^n$ and for $x,y \in I$ we have
	\begin{align}\label{eq:bd} \frac{f_{\a}'(x)}{f_{\a}'(y)} \leq \exp (B |x-y|) \end{align}
	Indeed, by the mean value theorem we have that
	\begin{align*}
	 \exp \Big( \log  \frac{f_{\a}'(x)}{f_{\a}'(y)} \Big) &\leq \exp | \log f_{\a}'(x)- \log f_{\a}'(y)| = \exp (|(\log f_{\a}')'(z)|\cdot|x-y|) \\
	& = \exp \bigg(\frac{|T''(f_{\a}x)|}{|T'(f_{\a}x)|^2}\cdot|x-y|\bigg)\leq \exp(B |x-y|).
	\end{align*}
	Using \eqref{eq:bd}, we obtain
	$$ \exp(-B|x_{\a_i}-t_i|) \leq\frac{f'_{\a_{i-1}\b_i}(t_i)}{e^{-2\lambda n}e^{2\lambda n}f'_{\a_{i-1}\b_i}(x_{\a_i})} \leq \exp(B|t_i-x_{\a_i}|) $$
	where the upper bound is direct, but the lower bound is achieved by swapping $x_{\a_i}$ and $t_i$ in the lemma. We also have that $|x_{\a_i}-t_i| \leq C_{\varepsilon,n}e^{-\lambda n}$ because both points are in $I_{\a_i}$. Hence using the definition of $\zeta_{i,\A}(\b_i)$ we have that
	$$ \exp(-BkC_{\varepsilon,n}e^{-\lambda n}) \leq \frac{f'_{\A\#\B}(t)}{e^{-2k\lambda n}\zeta_{1,\A}(\b_1)...\zeta_{k,\A}(\b_k)} \leq \exp(BkC_{\varepsilon,n}e^{-\lambda n}). $$
	Write
	$$P_k := e^{-2k\lambda n}\zeta_{1,\A}(\b_1)...\zeta_{k,\A}(\b_k).$$
	Then
	$$ [\exp(-BkC_{\varepsilon,n}e^{-\lambda n})-1]P_k \leq f'_{\A\#\B}(t)-P_k \leq [\exp(BkC_{\varepsilon,n}e^{-\lambda n})-1]P_k.$$
	So by integrating between $\hat{y}$ and $\hat{x}$ we get that
	\begin{align*} [\exp(-BkC_{\varepsilon,n}e^{-\lambda n})-1]P_k (\hat{y}-\hat{x}) &\leq f_{\A*\B}(x)-f_{\A*\B}(y)-P_k(\hat{y}-\hat{x})\\ &\leq [\exp(BkC_{\varepsilon,n}e^{-\lambda n})-1]P_k (\hat{y}-\hat{x}). \end{align*}
	Since $\hat{y},\hat{x} \in I_{\a_k}$ and $\zeta_{i,\A}\in [C_{\varepsilon,n}^{-2},C_{\varepsilon,n}^2]$, we have that 
	$$|P_k| \leq C_{\varepsilon,n}^ke^{-2k\lambda n}$$ 
	and so
	$$ |f_{\A*\B}(x)-f_{\A*\B}(y)-P_k(\hat{y}-\hat{x})| \leq e^{2k}C_{\varepsilon,n}^{k+2}e^{-(2k+2)\lambda n}.$$
	Fix $(x,y) \in [0,1] \times [0,1]$. Define
	$$\eta(x,y) := \xi e^{-2k \lambda n} (\hat{x}-\hat{y}).$$
	Then
	$$2\pi i \xi P_k(\hat{y}-\hat{x}) = 2\pi i \eta(x,y) \zeta_{1,\A}(\b_1)...\zeta_{k,\A}(\b_k)$$
	which, using $|\xi| \sim e^{(2k+1) n \lambda} e^{\eps_0 n}$, gives us
	$$|2\pi i \xi (f_{\A*\B}(x)-f_{\A*\B}(y))- 2\pi i \eta(x,y) \zeta_{1,\A}(\b_1)...\zeta_{k,\A}(\b_k)| \lesssim e^{2k}C_{\varepsilon,n}^{k+2}e^{-\lambda n} e^{\eps_0 n}$$
	
	By the mean value theorem and using Lemma \ref{lma:regularmarkov} we get that 
	$$C_{\varepsilon,n}^{-1}e^{-\lambda n} |x-y| \leq |\hat{x}-\hat{y}| \leq C_{\varepsilon,n} e^{-\lambda n} |x-y|$$ 
	and hence we have that
	$$ C_{\varepsilon,n}^{-1}e^{\eps_0 n} |x-y| \leq |\eta(x,y)| \leq C_{\varepsilon,n} e^{\eps_0 n} |x-y|. $$ 
	Splitting the sum we obtain
	\begin{align*}
	& \Big| \sum_{\B} e^{2\pi i \xi (f_{{\bf A*B}}(x) - f_{{\bf A*B}}(x))} \Big|\\
	&\leq \Big| \sum_{\B} e^{2\pi i \eta(x,y)\zeta_{1,{\a}}({\bf b_1})...\zeta_{k,{\a}}({\b}_k)} \Big| + \Big| \sum_{\B} e^{2\pi i \xi (f_{{\bf A*B}}(x) - f_{{\bf A*B}}(x))}-e^{2\pi i \eta(x,y)\zeta_{1,{\a}}({\bf b_1})...\zeta_{k,{\a}}({\b}_k)} \Big|
	\end{align*}
	Here
	\begin{align*} &\Big| \sum_{\B \in \cR_n^k(\eps)} e^{2\pi i \xi (f_{{\bf A*B}}(x) - f_{{\bf A*B}}(x))}-e^{2\pi i \eta(x,y)\zeta_{1,{\a}}({\bf b_1})...\zeta_{k,{\a}}({\b}_k)} \Big|\\
	&\leq  \sum_{\B \in \cR_n^k(\eps)} |2\pi \xi (f_{{\bf A*B}}(x) - f_{{\bf A*B}}(x))-2\pi  \eta(x,y)\zeta_{1,{\a}}({\bf b_1})...\zeta_{k,{\a}}({\b}_k)|\\
	& \lesssim  \sum_{\B \in \cR_n^k(\eps)} e^{2k}C_{\varepsilon,n}^{k+2}e^{-\lambda n} e^{\eps_0 n} \\
	& \lesssim e^{2k}C_{\varepsilon,n}^{k+2} e^{-\lambda n}e^{\eps_0 n} \sharp \cR_n^k(\eps)
	\end{align*}
	Combining the above estimates gives us
	\begin{align*} | \widehat{\mu}(\xi)| ^2 \lesssim\,\, & C_{\varepsilon,n}^{(2k+1)\lambda}e^{-\lambda (2k+1)\delta n}  \sum\limits_{{\A} \in \cR_n^{k+1}} \int_{I_{b(\A)}}\int_{I_{b(\A)}} \Big| \sum\limits_{{\B}} e^{2\pi i \eta(x,y) \zeta_{1,{\a}}({\bf b_1})...\zeta_{k,{\a}}({\b}_k)} \Big| \,d\mu(x)\,d\mu(y) \\
	&+e^{2k}C_{\varepsilon,n}^{k+2} e^{-\lambda n}e^{\eps_0 n} +\mu([0,1]\setminus R_n^{k+1}(\eps))^2+ e^{-\eps_2 n/2}
	\end{align*}
	By covering the $n$-regular part of the following set with $\lfloor \eps_0 n/2 \rfloor$-generation parent intervals, for fixed $y \in [0,1]$ we have that
	\begin{align*}
	\mu (\{x\in I : |x-y| \leq C_0 e^{-\eps_0 n/2}\}) & \leq \mu([0,1]\setminus R_n(\eps)) + \mu(\{ x \in R_n(\eps) : |x-y|\leq C_0 e^{-\eps_0 n/2} \}),
	\end{align*}
	where
	\begin{align*}
	 \mu(\{ x \in R_n(\eps) : |x-y|\leq C_0 e^{-\eps_0 n/2} \})& \lesssim C_0 C_{\varepsilon,n}e^{- \delta \eps_0 n/2}.
	\end{align*}
	Hence we have that
	$$ \mu \times \mu (\{(x,y)\in [0,1] \times [0,1] : |x-y| \leq C_0 e^{-\eps_0 n/2}\}) \lesssim_\mu \mu([0,1]\setminus R_n(\eps)) +C_0 C_{\varepsilon,n}e^{- \delta \eps_0 n/2}.$$
	Choose now $C_0 = C_{\eps,n}$. Now, when $|x-y| > C_{\eps,n} e^{-\eps_0 n/2}$, then
	$$e^{\eps_0 n/2}  = C_{\varepsilon,n}^{-1}e^{\eps_0 n} C_{\eps,n} e^{-\eps_0 n/2} \leq |\eta(x,y)|. $$
	Hence, when removing this measure and using triangle inequality and using the fact that the cardinalities of the sums over $\A$ and $\B$ is in total at most $Ce^{\lambda(2k+1) \delta n}$, we are left with the error
	$$\mu([0,1]\setminus R_n(\eps))+C_{\varepsilon,n}^2e^{- \delta \eps_0 n/2}$$
	and the double integral over those pairs $(x,y)$ with $|x-y| \geq C_{\eps,n} e^{-\eps_0 n/2}$. Hence when considering the supremum over $\eta \in J_n(\eps)$, we can bound over those $(x,y)$ with $|x-y| \geq C_{\eps,n} e^{-\eps_0 n/2}$.
\end{proof}

\section{Sum-product estimates}\label{sec:additive}

To control the exponential sums arising in Lemma \ref{lma:exponentialsum}, we will, as Bourgain and Dyatlov used in \cite{bourg}, use the following Fourier decay theorem for multiplicative convolutions proved in this form by Bourgain \cite[Lemma 8.43]{Bourgain2010} that follows from the discretised sum-product theorem. Recall that the multiplicative convolution of two measures $\mu$ and $\nu$ on $\R$ is defined by
$$\int f \, d(\mu \otimes \nu) = \iint f(xy) \, d\mu(x) \, d\nu(y), \quad f \in C_0(\R).$$

\begin{lemma}[Bourgain]\label{lma:keybourgain}
For all $\kappa > 0$, there exist $\eps_3 > 0$, $\eps_4  > 0$ and $k \in \mathbb{N}$ such that the following holds.

Let $\mu$ be a probability measure on $[\tfrac{1}{2}, 1]$ let and $N$ be a large integer. Assume for all $1/N < \rho < 1/N^{\eps_3}$ that
\begin{align}
\label{eq:boxdim}
\max_a \mu(B(a,\rho)) < \rho^\kappa.
\end{align}
Then for all $\xi \in \R$, $|\xi| \sim N$, the Fourier transform
\begin{align}
\label{eq:bourgfourier}
|\widehat{\mu^{\otimes k}}(\xi)| < N^{-\eps_4}.
\end{align}
\end{lemma}

Here $|\xi| \sim N$ means that there exists a constant $c > 0$ such that $c^{-1} N \leq |\xi| \leq cN$. In \cite[Proposition 3.2]{bourg} Bourgain and Dyatlov showed that by taking linear combinations of measures $\mu_j$, one can prove an analogous statement for multiplicative convolutions of several measures $\mu_j$ with the growth assumption \eqref{eq:boxdim} on $\R$ replaced with a growth assumption for $\mu_j \times \mu_j$ on $\R^2$. Then in the case of discrete measures $\mu_j$, this implies the following decay theorem for exponential sums:

\begin{lemma}[Bourgain-Dyatlov]\label{lma:bdexponential}
Fix $\delta_0 > 0$. Then there exist $k \in \N$, $\eps_2 > 0$, $\eps_3 > 0$ depending only on $\delta_0$ such that the following holds. Let $C_0, N \geq 0$ and $\cZ_1,\dots,\cZ_k$ be finite sets such that $\sharp \cZ_j \leq C_0 N$. Suppose $\zeta_j$, $j = 1,\dots,k$, on the sets $\cZ_j$ satisfy for all $j = 1,\dots,k$ that
\begin{itemize}
\item[(a)] the range
$$\zeta_j(\cZ_j) \subset [C_0^{-1},C_0];$$
\item[(b)] for all $\sigma \in [|\eta|^{-1},|\eta|^{-\eps_3}]$
$$\sharp\{(\b,\c) \in \cZ_j^2 : |\zeta_j(\b) - \zeta_j(\c)| \leq \sigma\} \leq C_0 N^2 \sigma^{\delta_0}.$$
\end{itemize}
Then for some constant $C_1$ depending only on $C_0$ and $\delta_0$ we have for all $\eta \in \R$, $|\eta| > 1$, that
$$\Big|N^{-k} \sum_{\b_1 \in \cZ_1,\dots,\b_k \in \cZ_k} \exp(2\pi i  \eta \zeta_1(\b_1) \dots \zeta_k(\b_k))\Big| \leq C_1 |\eta|^{-\eps_2}.$$
\end{lemma}

However, in our case, due to the $C_{\eps,n} := e^{\eps n}$ multiplicative fluctuations arising from large deviations of the distortion function $\tau = \log |T'|$, the maps $\zeta_j$ we obtain do not map the sets $\cZ_j$ into a fixed interval $[C_0^{-1},C_0]$, but when we increase $|\eta|$, the $C_0$ will change and will actually blow-up polynomially in $|\eta|$. Since the constant $C_1$ in Lemma \ref{lma:bdexponential} depends on $C_0$, it could cause problems when we increase $|\eta|$. For this reason we will open up the argument of Bourgain and Dyatlov (Proposition 3.2 of \cite{bourg}) to give a more precise dependence on the constant $C_1$ and $C_0$ and have the following quantitative version:

\begin{lemma}\label{lma:newexponential}
Fix $\eps_0 > 0$. Then there exist $k \in \N$, $\eps_2 > 0$, $\eps_3 > 0$ depending only on $\eps_0$ such that the following holds. Let $R,N > 1$ and $\cZ_1,\dots,\cZ_k$ be finite sets such that $\sharp \cZ_j \leq R N$. Suppose $\zeta_j$, $j = 1,\dots,k$, on the sets $\cZ_j$ satisfy for all $j = 1,\dots,k$ that
\begin{itemize}
\item[(1)] the range
$$\zeta_j(\cZ_j) \subset [R^{-1},R];$$
\item[(2)] for all $\sigma \in [R^{-2}|\eta|^{-1},|\eta|^{-\eps_3}]$
$$\sharp\{(\b,\c) \in \cZ_j^2 : |\zeta_j(\b) - \zeta_j(\c)| \leq \sigma\} \leq N^2 \sigma^{\eps_0}.$$
\end{itemize}
Then there exists a constant $c > 0$ depending only on $k$ such that we have for all $\eta \in \R$ with $|\eta|$ large enough, that
$$\Big|N^{-k} \sum_{\b_1 \in \cZ_1,\dots,\b_k \in \cZ_k} \exp(2\pi i  \eta \zeta_1(\b_1) \dots \zeta_k(\b_k))\Big| \leq c R^{k} |\eta|^{-\varepsilon_2}.$$
\end{lemma}

\begin{proof}
	We begin by altering assumption (2). We have that
	$$ \mu_j([x-\sigma, x+\sigma]) \leq \,\sigma^{\eps_0/2} $$
	for $\sigma \in [R^{-2}|\eta|^{-1},|\eta|^{-\varepsilon_2}/2]$ by using (2).
Define a measure $\mu_j$ on $\R$ by
$$\mu_j(A) = N^{-1} \sharp \{\b \in \cZ_j : \zeta_j(\b) \in A\}, \quad A \subset \R.$$
Then $\mu_j(\R) \leq R$ and by the assumptions (1) and (2) of the lemma we are about to prove, we have that the measure $\mu_j$ is a Borel measure on $[R^{-1},R]$ and that
$$(\mu_j \times \mu_j)(\{(x,y) \in \R^2 : |x-y| \leq \sigma\}) \leq \sigma^{\eps_0}$$
for all $\sigma \in [R^{2}|\eta|^{-1},|\eta|^{-\eps_2}]$. Then to prove the claim, we just need to check that the Fourier transform of the multiplicative convolutions of $\mu_j$ satisfies:
$$|(\mu_1 \otimes \dots \otimes \mu_k)\,\,\hat{ }\,\,(\eta)| \leq R^k|\eta|^{-\varepsilon_2}.$$
The rate of decay to be found will be given by
$$ \varepsilon_2:= \frac{1}{10}\min(\varepsilon_4,\varepsilon_3) $$
where $\varepsilon_3$ and $\varepsilon_4$ are given in Lemma \ref{lma:keybourgain}.

Fix $\ell \in \N$ such that $2^{\ell} < R \leq 2^{\ell+1}$. Then $\supp \mu_j \cap [R^{-1},R]$ can be covered by intervals of the form $I^{[i]}:=[2^{i-1},2^i]$ for $i=-l,...,l,l+1$. Let $\mu_j^{[i]}$ be $\mu_j$ restricted to $I^{[i]}$. Thus writing the re-scaling map
$$S_{i}(x) = 2^{-i}x, \quad x \in \R,$$
we have that the measure $\nu_j^{[i]} = S_{i}(\mu_j^{[i]})$ is supported on $[\tfrac{1}{2},1]$. Moreover, it satisfies
\begin{align*}
(\nu_j' \times \nu_j')(\{(x,y) \in \R^2 : |x-y| \leq \sigma\}) & \leq (\mu_j \times \mu_j)(\{(x,y) \in \R^2 : |x-y| \leq 2^i\sigma\})\\
& \leq (2^i\sigma)^{\varepsilon_0}\leq \sigma^{\varepsilon_0/2}
\end{align*}
where we use the fact that $2R \leq \sigma^{-1/2}$ which holds by assuming that $|\eta| > 4$. We know that the main assumption is satisfied for $\sigma \in [2^i |\eta|^{-1}, 2^i |\eta|^{\varepsilon_3}]$, so for the rescaled measure, we get the main assumption for the required range $\sigma \in [|\eta|^{-1},|\eta|^{\varepsilon_3}]$
 We will use that fact that
$$(\nu_1^{[i_1]} \otimes \dots \otimes \nu_k^{[i_k]})\,\,\hat{ }\,\,\Big(\eta\prod_{j=1}^{k}2^{-i_j}\Big) = (\mu_1^{[i_1]} \otimes \dots \otimes \mu_k^{[i_k]})\,\,\hat{ }\,\,(\eta).$$
Each $\mu_j$ is a sum of at most $2l+2$ of the restricted measures $\mu_j^{[i]}$, so the Fourier transform $(\mu_1\otimes \dots \otimes \mu_k)\,\,\hat{ }\,\,(\eta)$ decomposes into at most $(2l+2)^k$ terms consisting of Fourier transforms $(\mu_{1}^{[i_1]} \otimes \dots \otimes \mu_{k}^{[i_k]})\,\,\hat{ }\,\,(\eta)$ going through all the possible restrictions $\mu_j^{[i]}$. Hence if we can prove
$$|(\nu_1^{[i_1]} \otimes \dots \otimes \nu_k^{[i_k]})\,\,\hat{ }\,\,(\eta)| \leq C^*|\eta|^{-\varepsilon_2}$$
for some constant $C^* > 0$ only depending on $k$, the triangle inequality gives
$$|(\mu_1\otimes \dots \otimes \mu_k)\,\,\hat{ }\,\,(\eta)| \leq 2(2l+1)^k C^*|2^{-lk}\eta|^{-\varepsilon_2}\lesssim R^{k}C^*|\eta|^{-\varepsilon_2} $$

Thus let us assume from the start that $\mu_j$ is supported on $[\tfrac{1}{2},1]$. As in \cite{bourg}, let us first argue that it is enough to consider the case $\mu_1 = \mu_2 = \dots = \mu_k$. Given $\lambda = (\lambda_1,\dots,\lambda_k) \in [0,1]^k$, write
$$G(\lambda) := (\mu_{\lambda} \otimes \dots \otimes \mu_{\lambda})\,\,\hat{ }\,\,(\eta) = \widehat{\mu_\lambda^{\otimes k}}(\eta).$$
and the linear combination 
$$\mu_\lambda = \lambda_1 \mu_1 + \dots + \lambda_k \mu_k.$$
Expanding $\widehat{\mu_\lambda^{\otimes k}}(\eta)$ using the definition of $\mu_\lambda$ as a weighted sum of $\mu_k$'s, we see that it contains at most $k^k$ terms involving multiplicative convolutions of $\mu_j$ with coefficients given by products of $\lambda_1,\dots,\lambda_k$. 
Then if we know the claim for $\mu_1 = \dots = \mu_k$, then we can apply it to $\mu_\lambda$ and obtain 
$$\sup_{\lambda \in [0,1]^k} |G(\lambda)| \leq |\eta|^{-\varepsilon_2}.$$
From this we see that as the map $G$ is a polynomial of degree $k$, so there is a constant $C^* > 0$ depending on $k$
$$\frac{1}{k!} |\partial_{\lambda_1} \dots \partial_{\lambda_k} G(\lambda)|_{\lambda = 0}| \leq C^*|\eta|^{-\varepsilon_2}.$$
However,
$$|(\mu_1 \otimes \dots \otimes \mu_k)\,\,\hat{ }\,\,(\eta)| = \frac{1}{k!} |\partial_{\lambda_1} \dots \partial_{\lambda_k} G(\lambda)|_{\lambda = 0}|,$$
so this gives the claim.

As for the case $\mu_1 = \mu_2 = \dots = \mu_k$, depending on the amount of mass $\mu_1$ has, we have two cases. 

If $\mu_1(\R) \geq |\eta|^{-\eps_3 \eps_0/10}$, choose an integer $N$ such that $N/2 \leq |\eta| \leq N$. The probability measure
$$\mu_0 = \frac{\mu_1}{\mu_1(\R)}$$
on $\R$ satisfies
$$\sup_x \mu_0(B(x,\sigma)) < \sigma^{\eps_0/2}$$
for all $\sigma \in [4R^{-2}N^{-1},N^{-\eps_3}]$. Similarly we have by applying the above for $\sigma := 4R^2\sigma$ (when $R>1$), we obtain this for $\sigma \in [N^{-1},4R^2N^{-\varepsilon_3}]$ by monotonicity of $\mu$, which holds for $|\eta|^{1-\varepsilon_3}\geq 16R^4$. Hence Lemma \ref{lma:keybourgain} proves the claim. Note that here the constant dependence does not change.

If $\mu_1(\R) \leq |\eta|^{-\eps_3 \eps_0/10}$, then one can use a trivial bound on exponential function in the integral convolution and triangle inequality to obtain the claim. The desired decay can be achieved by noting $k\geq1$ in this final case.
\end{proof}

\section{Total non-linearity, spectral gap and non-concentration} \label{sec:nonconcentration} In order to apply the sum-product estimate (Lemma \ref{lma:newexponential}) in our setting, we will need to verify a suitable non-concentration condition for the maps $\zeta_j = \zeta_{j,A} : \cR_n(\eps) \to \R$ which are defined by
$$\zeta_{j,\A}(\b) := e^{2\lambda n} f_{\a_{j-1} \b}'(x_{\a_j}), \quad \b \in \cR_n(\eps),$$
where $\A = \a_0 \a_1\dots \a_k \in \cR_n^{k+1}$ and $x_{\a_j}$ is the center point of the interval $I_{\a_j}$. This condition will be possible thanks to a contraction theorem for complex transfer operators on the Banach space of Lipschitz functions, which we will now state.

Let $C^{\mathrm{Lip}}(K)$ be the set of all complex valued Lipschitz functions $g$ on $K$. For $b \in \R$, $b \neq 0$, define the Lipschitz $b$-norm
$$\|g\|_{\mathrm{Lip},b} := \|g\|_\infty + \frac{\mathrm{Lip}(g)}{|b|},$$
where $\mathrm{Lip}(g)$ is the optimal Lipschitz constant of $g$. Given a function $\psi : I = \bigcup_{a = 1}^N I_a \to \C$, define the transfer operator $\mathcal{L}_{\psi}$ on the Banach space $C^{\mathrm{Lip}}(K)$ by
$$\mathcal{L}_{\psi} g(x) :=  \sum_{y : T(y) = x} e^{\psi(y)} g(y)$$
If $s \in \C$, then in the case of the potential $\psi = \phi-s \tau$ for the distortion function $\tau = \log |T'|$, the following was proved by Stoyanov in \cite{Stoyanov}:

\begin{thm}[Contraction of transfer operators]\label{thm:c1cont}
Under the assumptions of Theorem \ref{thm:nonlinear}, the following holds. Let $0 < \Xi < 1$. Then there exists $C_\Xi > 0$, $\delta_1(\Xi) > 0$, $t_0(\Xi) > 0$ such that for all $s \in \C$ of the form $s = \delta - ib$, where and $|\mathrm{Im\,}s| = |b| \geq t_0(\Xi)$ we have for all $f \in C^{\mathrm{Lip}}(K)$ and $m \in \N$ that
	$$  \|\cL_{\phi-s\tau}^m f \|_{\mathrm{Lip},b}\leq C_{\Xi}|\mathrm{Im\,} s|^{\Xi}e^{-\delta_1(\Xi) m}  \|f\|_{\mathrm{Lip},b},$$
	where $\delta = \Hd \mu > 0$ is the unique real number satisfying $P(\phi-\delta \tau) = 0$ and $P$ is the topological pressure on $K$. 
\end{thm}

Theorem \ref{thm:c1cont} is a generalisation of the contraction theorem of Naud \cite[Theorem 2.3]{Naud}, where Naud proved his result in the case of $\phi = 0$. This result works better with the Gibbs measures we are dealing with. This version is given in \cite[Theorem 1.1 and Theorem 5.1]{Stoyanov}. However, the statement of Theorem \ref{thm:c1cont} in \cite{Stoyanov} was done in the language of $C^2$ Axiom A flows on $C^2$ complete Riemann manifolds, where a \textit{Local Non-Integrability Condition} (LNIC) is assumed from the flow and the $\tau$ is a first return time function for the flow. Using Markov partitions of rectangles and projections along unstable manifolds, see e.g. \cite[Definition 2.6]{AGY}, one can realise the Anosov flow as a suspension flow over the roof function $\tau$. We will apply this in the case of the roof function $\tau = \log |T'|$.

Stoyanov points out that LNIC holds if the Anosov flows satisfies the well-known \textit{Uniform Non-Integrability} (UNI) for Anosov flows by Chernov \cite{Chernov}, which means that the means that the roof function $\tau$ for the Anosov flow has oscillations. Here we use the following symbolic definition, which is also used in \cite[(7.6) and Definition 2.3]{AGY}, Baladi-Vall\'ee \cite{BV} and Naud \cite[Proposition 5.5]{Naud} and in the notation of this article goes as follows.

\begin{definition} We say that the function $\tau$ satisfies \textit{Uniform Non-Integrability} (UNI) if there exists $c > 0$ and $n_0 \in \N$ such that for all $n \geq n_0$ there exists $\a,\b \in \cA^n$ such that for all $x \in K$:
\begin{align} \label{eq:asymp}\Big| \frac{d}{dx} (S_n \tau(f_\a(x)) - S_n \tau(f_\b(x)) )\Big| \geq c.\end{align}
\end{definition}

Recall that $T$ is totally non-linear if it is not possible to write 
$$\tau = \psi_0 + g - g \circ T$$ 
on $I$ for $\psi_0 : I \to \R$ constant on every $I_a$, $a \in \cA$, and $g \in C^1(I)$. The following lemma connects total non-linearity to UNI:

\begin{lemma}\label{lma:NLI}
\begin{itemize}
 \item[(1)] If $\bigcup_{a \in \cA} I_a = [0,1]$ and $f_a$ are $C^2$, then $T$ is totally non-linear if and only if $\tau = \log |T'|$ satisfies \emph{UNI}.
\item[(2)] If $\{I_a : a \in \cA\}$ are disjoint and $f_a$ are analytic, then $T$ is totally non-linear if and only if $\tau = \log |T'|$ satisfies \emph{UNI}.
\end{itemize}
\end{lemma}

The proof Lemma \ref{lma:NLI}(1) was done for example by Avila, Gou\"ezel, Yoccoz  \cite[Proposition 7.5]{AGY}, where the statement is for expanding $C^2$ Markov maps on John domains, which is our setting in dimension one when $\bigcup_{a = 1}^N I_a= [0,1]$. This same proof goes also through in the case (2) where $\{I_a : a \in \cA\}$ are disjoint if we assume $f_a$ are analytic. This was also proved in Naud \cite[Lemma 4.3 and Proposition 5.1]{Naud} using the notion of \textit{Non-Local Integrability (NLI)}. The proof that the total non-linearity is equivalent to \eqref{eq:asymp} goes back to the notion of \textit{Anosov alternative} by Dolgopyat \cite{Dolgopyat,Dolgopyat1}. Thus we can apply Theorem \ref{thm:c1cont} under the assumptions of Theorem \ref{thm:nonlinear}.

Let us now use Theorem \ref{thm:c1cont} to prove the main non-concentration estimate needed for the sum-product estimates. Before that we will need to fix the parameters $R$ and the range of $\sigma$ we consider needed for Lemma \ref{lma:newexponential}:

\begin{remark}\label{rmk:parameters}
\begin{itemize}
\item[(1)] 
For $n \in \N$ and $\eps > 0$ the number
$$R = R(n,\eps) := 16^2CC_{\varepsilon,n}^{3\lambda},$$
where $C > 0$ is the constant satisfying the Gibbs condition of $\mu$, recall \eqref{eq:Gibbs} and recall $\lambda = \int \tau \, d\mu$ is the Lyapunov exponent of $\mu$. Then for the map
$$ \zeta_{j,{\A}}({\b}) = e^{2\lambda n}f_{{\bf a_{j-1}b}}'(x_{{\bf a_j}}). $$
we see that
$$\zeta_{j,\A}(\b) \in [R^{-1},R].$$
Indeed, the chain rule gives
$$ \zeta_{j,{\A}}({\b}) = e^{2\lambda n}f_{{\bf a_{j-1}}}'(f_{\b}x_{{\bf a_j}})f_{\b}'(x_{\a_j}) $$
so we can apply Lemma \ref{lma:regularmarkov} and the fact that $f'_{\a_{j-1}}$ and $f_{\b}'$ must both be either positive or negative because they are defined by words of the same length. 
\item[(2)] Fix now any $\Xi \in (0,1)$ and let $\delta_1(\Xi) > 0$ be the exponent from Theorem \ref{thm:c1cont}. Define
$$\eps_0 := \min\{\delta_1(\Xi)/2,\lambda/4\} > 0.$$
Then $\eps_0 < \lambda/2$ and this now fixes the regular words 
$$\cR_n(\eps) = \cR_n(\eps,\eps_0),$$
blocks 
$$\cR_n^k(\eps) = \cR_n^k(\eps,\eps_0)$$ 
and parameters
$$J_n(\eps) = J_n(\eps,\eps_0) = \{ \eta \in \mathbb{R}:e^{\eps_0 n/2} \leq |\eta| \leq C_{\varepsilon,n} e^{\eps_0 n}  \},$$
from Proposition \ref{lma:exponentialsum}, which all implicitly depend on $\eps_0 > 0$.
\end{itemize} 
\end{remark}

\begin{prop}\label{lma:distribution}
Under the assumptions of Theorem \ref{thm:nonlinear}, the following holds. Let $\eps_3 > 0$ is from Lemma \ref{lma:newexponential}. Write $\mathcal{W} \subset \mathcal{R}_n^{k+1}(\eps)$ to be the set of ``well-distributed blocked words'' ${\A}$ defined such that for all $j=1,\dots,k$, $\eta \in J_n(\eps)$ and $\sigma \in [R^{-2}|\eta|^{-1},|\eta|^{-\eps_3}]$, where we have that
	$$ \sharp \{ ({\bf b, c}) \in \mathcal{R}_n(\eps)^2: |\zeta_{j,{\A}}({\b})-\zeta_{j,{\A}}({\c})| \leq \sigma \}  \leq \sharp \cR_n(\eps)^2 \sigma^{c_0/2}. $$
Then most blocks are well-distributed, so for some $\kappa_0>0$,
	$$ e^{-\lambda (k+1)\delta n}|\mathcal{R}_n^{k+1}(\eps)\setminus \mathcal{W}| \leq C_{\varepsilon,n}^{2\kappa_0} \sigma^{c_0/4}.$$
\end{prop}

We will now give the key non-concentration estimate for distortions as a consequence of Theorem \ref{thm:c1cont}.

\begin{lemma}[Non-concentration]\label{thm:nonconc} There exists $c_0 > 0$ and $\kappa_0 > 0$ such that for all $\eps > 0$, $n \in \N$, $\eta \in J_n(\eps)$, $\sigma \in [R^{-2}|\eta|^{-1},|\eta|^{-\eps_3}]$, $x \in [0,1]$ we have
$$\sharp\{(\a,\b,\c) \in \cR_n(\eps)^3 : |e^{2\lambda n}f_{{\bf ab}}'(x)-e^{2\lambda n}f_{{\bf ac}}'(x)| \leq \sigma \} \lesssim C_{\eps,n}^{\kappa_0} \sigma^{c_0} \sharp \cR_n(\eps)^3,$$
Here $R$ is fixed in Remark \ref{rmk:parameters} and $\eps_3 > 0$ is from Lemma \ref{lma:newexponential}.
\end{lemma}

\begin{proof} Fix $\Xi > 0$ and let $\eps_0 = \min\{\delta_1(\Xi)/2,\lambda/4\}$ be the constant defined in Remark \ref{rmk:parameters}(2), where $\delta_1(\Xi) > 0$ comes from Theorem \ref{thm:c1cont}. For this $\Xi$, write also $t_0 = t_0(\Xi) > 0$ from Theorem \ref{thm:c1cont}. 

Choose $m \in \N$ such that $e^{-\eps_0 (m-1)}\leq \sigma \leq e^{-\eps_0 m}$. We will first prove that for all $y \in \R$ with $y\pm e^{-\eps_0 m} \in [R^{-1},R]$, $\e \in \tilde \cR_{2n-m}(\eps)$ and $x \in I$ we have
\begin{align} \label{eq:scalem}\sharp\{\d \in \tilde \cR_m(\eps) : e^{2\lambda n}f_{{\bf e d}}'(x) \in B(y,e^{-\eps_0 m})\} \lesssim  C_{\eps,m}^{\kappa_0} e^{-c_0 m} \sharp \tilde\cR_m(\eps),\end{align}
where
$$\tilde \cR_m(\eps) := \Big\{\d \in \cA^m :  I_\d \subset A_m(2\eps) \Big\} \quad \text{and} \quad \tilde \cR_{2n-m}(\eps) := \Big\{\e \in \cA^{2n-m} : I_\e \subset A_{2n-m}(2\eps) \Big\}.$$
Then the cardinalities
$$\sharp \tilde \cR_m(\eps) \sim C_{2\eps,n} \sharp \cR_m(\eps) \quad \text{ and } \quad \sharp\tilde \cR_{2n-m}(\eps) \sim C_{2\eps,n} \sharp\cR_{2n-m}(\eps)$$ 
by Lemma \ref{lma:regularmarkov} for the properties of regular words, where by $a \sim b$ we mean $b/c \leq a \leq cb$ for some constant $c > 0$. 

Lemma \ref{thm:nonconc} follows now from \eqref{eq:scalem} by first setting $\cP$ to be the set of pairs $(\e,y)$ such that $y\pm e^{-\eps_0 m} \in [R^{-1},R]$ and $\e \in \tilde \cR_{2n-m}(\eps)$, and bounding
\begin{align*}& \sharp\{(\a,\b,\c) \in \cR_n(\eps)^3 : |e^{2\lambda n}f_{{\bf ab}}'(x)-e^{2\lambda n}f_{{\bf ac}}'(x)| \leq \sigma \} \\
& \leq \sharp \cR_n(\eps) \sup_{\c \in \cR_n(\eps)} \sharp\{(\a,\b) \in \cR_n(\eps)^2 : |e^{2\lambda n}f_{{\bf ab}}'(x)-e^{2\lambda n}f_{{\bf ac}}'(x)| \leq \sigma \} \\
& \leq \sharp \cR_n(\eps) \sharp \tilde \cR_{2n-m}(\eps)  \sup_{(\e,y)  \in \cP} \sharp\{\d \in \tilde \cR_m(\eps) : e^{2\lambda n}f_{{\bf \e \d }}'(x) \in B(y, \sigma) \}
\end{align*}
 since every $\a\b$, for $\a,\b \in \cR_n(\eps)$ splits into a word $\a\b = \e \d$ with $\e := \a\b|_{2n-m} \in \tilde\cR_{2n-m}(\eps)$ and $\d := \sigma^{2n-m} (\a\b) \in \tilde \cR_m(\eps)$ using the quasi-Bernoulli property of the Gibbs measure $\mu$: since $\a\b = \e \d$, we have
$$\mu(I_{\e})\mu(I_\d) \lesssim \mu(I_{\a\b}) \lesssim \mu(I_{\e})\mu(I_\d)$$
and that the lengths
$$|I_{\e}| |I_\d| \lesssim |I_{\a\b}| \lesssim |I_{\e}||I_\d|.$$

Then fix $(\e,y) \in \cP$. Since $B(y,\sigma) \subset B(y,e^{-\eps_0 m})$ we have by \eqref{eq:scalem} that
 \begin{align*}
 \sharp\{\d \in \tilde \cR_m(\eps) : e^{2\lambda n}f_{{\bf \e \d }}'(x) \in B(y, \sigma) \} \lesssim C_{\eps,m}^{\kappa_0} e^{-c_0 m} \sharp \tilde\cR_m(\eps)
 \end{align*}
 Note that $e^{-c_0 m} = e^{-\frac{c_0}{\eps_0} \eps_0 m} \lesssim \sigma^{c_0/\eps_0}$ so the claim follows by the cardinality bounds of $\tilde \cR_{2n-m}(\eps)$ and $\tilde \cR_m(\eps)$ by setting in the statement of the lemma the exponent $c_0 > 0$ as $c_0/\eps_0$.

Let us now verify the non-concentration estimate \eqref{eq:scalem} we need.

\textbf{Step 1.}  Write $r := e^{-\eps_0 m}$. Since $y-r \geq R^{-1} > 0$, we know that  $e^{2\lambda n}f_{{\e \d}}'(x) \in B(y,r)$ if and only if 
$$-\log |f_{\e\d}'(x)| \in J := [2 \lambda n - \log(y+r),2\lambda n - \log(y-r)].$$ 
Note that the interval $J$ has length $|J| = \log\frac{y+r}{y-r}$. By the mean value theorem, we have
\begin{align}\label{eq:Jbound}2 R^{-1} r  \leq \frac{2r}{y+r} \leq |J| \leq \frac{2r}{y-r} \leq 2R r.\end{align}
Hence 
$$ \sharp\{\d \in \tilde \cR_m(\eps) : e^{2\lambda n}f_{{\e \d}}'(x) \in B(y,e^{-\eps_0 m})\} =  \sharp\{\d \in \tilde \cR_m(\eps) : -\log |f_{\e\d}'(x)| \in J\}$$
\textbf{Step 2.} Now let us approximate the indicator of $\chi_J$ by a mollifier $h \in C^2(\R)$ satisfying
\begin{itemize}
\item[(1)] $\chi_J \leq h$
\item[(2)] $\|h\|_1 \lesssim |J|$
\item[(3)] $\|h''\|_{L^1} \lesssim \frac{1}{|J|}.$
 \end{itemize}
This function can be obtained, for example using a scaled and translated Gaussian function
$$h(x) := e^{\frac{\pi}{4}} g_0\Big(\frac{x-x_J}{|J|}\Big), \quad \text{where} \quad g_0(x) := e^{-\pi x^2}.$$
where $x_J$ is the central point of $J$.

Since $\chi_J \leq h$, we have
$$ \sharp\{\d \in \tilde\cR_m(\eps) :  -\log |f_{\e\d}'(x)| \in J\}\leq \sum_{\d \in \tilde\cR_n(\eps) }h(-\log |f_{\e\d}'(x)|)^{1/2}$$
Recall that for $\d \in \cA^m$, we defined the Birkhoff weights $w_\d(x) = e^{S_n\phi(f_\d(x))}$. Using Cauchy-Schwartz inequality, we thus obtain
	\begin{align*} &\Big(\sum_{\d \in \tilde\cR_m(\eps) }h(-\log |f_{\e\d}'(x)|)^{1/2}\Big)^2 \\
	&\leq \Big(\sum_{\d \in \tilde\cR_m(\eps) } w_\d(x) |f_{\e\d}'(x)|^{\delta}h(-\log |f_{\e\d}'(x)|)\Big)\Big(\sum_{\d \in \tilde\cR_m(\eps) }\frac{1}{w_\d(x) |f_{\e\d}'(x)|^\delta}\Big) \end{align*}

\textbf{Step 3.} Taking the inverse Fourier transform of $\widehat{h}$ gives us for all $x \in I$, $\d \in \cA^m$ and $m \in \N$ that
\begin{align*}
h(-\log |f_{\e\d}'(x)|)=\int e^{-2\pi i \xi \log |f_{\e\d}'(x)|}\hat{h}(\xi)\,d\xi
\end{align*}
Therefore
\begin{align*} \sum_{\d \in \tilde \cR_m(\eps)} w_\d(x) |f_{\e\d}'(x)|^\delta h(-\log |f_{\e\d}'(x)|) &= \sum_{\d \in \tilde \cR_m(\eps)} w_\d(x) |f_{\e\d}'(x)|^\delta \int \hat{h}(\xi)e^{-2\pi i \xi \log |f_{\a\b}'(x)|}\,d\xi\\
&=\int \hat{h}(\xi)\sum_{\d \in \tilde \cR_m(\eps)} w_\d(x) |f_{\e\d}'(x)|^\delta e^{-2\pi i \xi \log |f_{\e\d}'(x)|}\,d\xi 
\end{align*}
Split the integration now over $|\xi| > t_0 / 2\pi$ and $|\xi| \leq t_0 / 2\pi$, were $t_0 = t_0(\Xi) > 0$ is from Theorem \ref{thm:c1cont} needed for the contraction of the transfer operator $\cL_{\phi-s\tau}$ in the range of $s \in \C$ with $|\mathrm{Im} \,s| \geq t_0(\Xi)$.

\textbf{Step 4.} If $|\xi| > t_0 / 2\pi$, we will estimate as follows. Inside the integral, use the estimate
$$\sum_{\d \in \tilde\cR_m(\eps) } w_\d(x) |f_{\e\d}'(x)|^\delta h(-\log |f_{\e\d}'(x)|)\leq \sum_{\d \in \cA^m} w_\d(x) |f_{\e\d}'(x)|^\delta h(-\log |f_{\e\d}'(x)|) $$
We can iterate the definition of the complex transfer operator applied for $g_\e$ defined by
$$g_{\e}(z) := |f'_{\e}(z)|^s, \quad z \in [0,1]$$
to obtain with $s = \delta-2\pi \xi i$ that
\begin{align}\label{eq:iterate}\mathcal{L}_{\phi-s \tau}^m g_\e(x) = \sum_{\d \in \cA^m} g_{\e}(f_{\d}(x))  e^{S_n\phi(f_\d(x)) + s \log |f_{\d}'(x)|} = \sum_{\d \in \cA^m} w_\d(x) |f_{\e\d}'(x)|^\delta e^{-2\pi i \xi \log |f_{\e\d}'(x)|}\end{align}
whenever $x \in [0,1]$ since $|f'_{\e}(z)|^s = |f'_{\e}(z)|^\delta e^{-2\pi i \xi \log |f'_{\e}(z)|}$ and using the chain rule.
Here we will employ the contraction of transfer operator $\cL_{\phi - s\tau}$, Theorem \ref{thm:c1cont} with first a fixed $0 < \Xi < 1$ to obtain 
$$|\mathrm{Im \,} s| = 2\pi |\xi| \geq t_0(\Xi)$$ 
by the choice of $\xi$. Thus we have for some $\delta_1(\Xi) > 0$ that

\begin{align*}
&\int\limits_{|\xi|>t_0/2\pi} \hat{h}(\xi)\sum_{\d \in \cA^m} w_\d(x) |f_{\e\d}'(x)|^\delta e^{-2\pi i \xi \log |f_{\e\d}'(x)|}\,d\xi\\
& \leq \int\limits_{|\xi|>t_0/2\pi} |\hat{h}(\xi)| \cdot ||\cL_{\phi-s\tau}^m g_\e ||_{\infty}\,d\xi\\
& \leq \int\limits_{|\xi|>t_0/2\pi} |\hat{h}(\xi)| \cdot ||\cL_{\phi-s\tau}^m g_\e ||_{\mathrm{Lip},2\pi \xi}\,d\xi\\
& \leq \int\limits_{|\xi|>t_0/2\pi} |\hat{h}(\xi)| \cdot C_{\Xi}|\mathrm{Im\,} s|^{\Xi} e^{-\delta_1(\Xi) m}  \|g_\e\|_{\mathrm{Lip},2\pi \xi} \,d\xi \\
& \lesssim_\Xi C_{\eps,n} e^{-\lambda \delta (2n-m)} e^{-\delta_1(\Xi) m} \int\limits_{|\xi|>t_0/2\pi} |\hat{h}(\xi)| \cdot |\xi|^{\Xi}\,d\xi.
\end{align*}
Here we used the bounded distortion for $T$ to bound the Lipschitz norm of $g_\e$. First of all, we have 
$$|g_\e(z)|= |f_\e(z)|^\delta \leq C_{\eps,n} e^{-\lambda \delta (2n-m)} \quad \text{and} \quad |g_\e'(z)| \lesssim |\xi||f_\e(z)|^\delta \leq  C_{\eps,n} e^{-\lambda \delta (2n-m)} |\xi|$$
so
$$ \|g_\e\|_{\mathrm{Lip},2\pi \xi} = \|g_\e\|_\infty + \frac{\mathrm{Lip}(g_\e)}{2\pi |\xi|} \lesssim C_{\eps,n} e^{-\lambda \delta (2n-m)}.$$ 
Indeed, after fixing a branch of the logarithm, using bounded distortion $|f''_\e(z)| \leq B |f'_\e(z)|$, and $|s| \lesssim |\xi|$, we obtain that
$$|g_\e'(z)| = \frac{|s|}{|f'_\e(z)|} |\exp(s \log |f'_\e(z)|)| |f''_\e(z)| = \frac{|s||T''(f_{\e}z)|}{|T'(f_{\e}z)|^2} |g_\e(z)| \lesssim |\xi| |g_\e(z)|$$
where we use the inverse rule for differentiable functions and the chain rule.

Using integration by parts, we have for the Fourier transform $\hat{g}$ that for all $\xi \in \R$ the following estimate holds:
\[|\hat{h}(\xi)|\leq \frac{1}{1+|2\pi \xi|^2}(\|h\|_{L^1}+\|h''\|_{L^1}) .\]
Then in particular as $\Xi < 1$, we have
\[\int_{|\xi|>t_0 / 2\pi} |\hat{h}(\xi)| \cdot |\xi|^{\Xi}\,d\xi\leq \int \frac{\xi^{\Xi}}{1+|2\pi \xi|^2}(\|h\|_{L^1}+\|h''\|_{L^1}) \,d\xi \lesssim \|h\|_{L^1}+\|h''\|_{L^1}. \]

\textbf{Step 5.} We are left with the case $|\xi| \leq t_0 / 2\pi$, that is, an estimation for
$$\int_{|\xi| \leq t_0 / 2\pi} \hat{h}(\xi)\sum_{\d \in \tilde \cR_m(\eps)} w_\d(x) |f_{\e\d}'(x)|^\delta e^{-2\pi i \xi \log |f_{\e\d}'(x)|}\,d\xi$$
Let us bound
\begin{align*}
\sup_{|\xi|\leq t_0/2\pi} \Big|\hat{h}(\xi)\sum_{\d \in \tilde\cR_m(\eps)} w_\d(x) |f_{\e\d}'(x)|^\delta e^{-2\pi i \xi \log |f_{\e\d}'(x)|} \Big|& \leq 
\sup_{|\xi|\leq t_0/2\pi} |\hat{h}(\xi)|\sum_{\d \in \tilde\cR_m(\eps)} w_\d(x) |f_{\e\d}'(x)|^\delta\\
& \leq C_{\eps,2n-m} e^{-2 \lambda \delta n} |J|  
\end{align*}
since $|\hat{h}(\xi)|\leq \|h\|_{L^1}\leq 3|J|$ and by the chain rule
\begin{align*}\sum_{\d \in\tilde\cR_m(\eps)} w_\d(x) |f_{\e\d}'(x)|^\delta &= \sum_{\d \in \tilde\cR_m(\eps)} w_\d(x)  |f_{\e}'(f_\d(x))|^\delta |f_\d'(x)|^\delta\\ &\lesssim C_{\eps,2n-m} e^{-\lambda \delta (2n-m)}\sum_{\d \in \tilde\cR_m(\eps)} w_\d(x)  |f_\d'(x)|^\delta \\
& \lesssim C_{\eps,2n-m} e^{-\lambda \delta (2n-m)} e^{-\lambda \delta m}\\
& = C_{\eps,2n-m} e^{-2\lambda \delta n}
\end{align*}
by the properties of $\tilde \cR_m(\eps)$.

\textbf{Step 6.} Combining Step 1, 2, 3, 4 and 5 gives us
\begin{align*}
	&\sharp\{\d \in \tilde \cR_m(\eps) : e^{2\lambda n}f_{{\e \d}}'(x) \in B(y,e^{-\eps_0 m})\}^2 \\
	&\lesssim_\Xi C_{\eps,2n-m}e^{-\lambda \delta (2n-m)}  E_\e(x) [e^{-\delta_1(\Xi) m} (\|h\|_{L^1}+\|h''\|_{L^1})+ |J|],
\end{align*}
where
$$E_\e(x) := \sum_{\d \in \tilde\cR_m(\eps) }\frac{1}{ w_\d(x) |f_{\e\d}'(x)|^\delta}.$$
Finally, let us now analyse all the quantities we have. Lemma \ref{lma:regularmarkov} gives that
$$E_\e(x) \lesssim C_{\eps,2 m}^\delta e^{2\lambda \delta n} e^{\lambda \delta m} \sharp \tilde \cR_m(\eps)$$
so for some $\kappa > 0$
$$C_{\eps,2n-m}e^{-2\lambda \delta n} E_\e(x) \lesssim C_{\eps,n}^\kappa \sharp \tilde \cR_m(\eps)^2$$
Moreover, recall that \eqref{eq:Jbound} gives
$$ |J| \leq 2R r$$
and
$$\frac{1}{|J|^3} \leq \frac{1}{2} R^{3} r^{-3}$$
and when inputting $r = e^{-\eps_0 m}$ and $R = 16^2CC_{\varepsilon,m}^{3\lambda}$, we obtain
$$ |J| \lesssim C_{\varepsilon,m}^{3\lambda} e^{-\eps_0 m}$$
and
$$\frac{1}{|J|} \lesssim  C_{\varepsilon,m}^{3\lambda} e^{\eps_0 m}$$
Then by the choice of $h$, we have
$$\|h\|_{L^1}+\|h''\|_{L^1} \leq |J| + \frac{1}{|J|} \lesssim C_{\varepsilon,m}^{3\lambda} e^{-\eps_0 m} + C_{\varepsilon,m}^{3\lambda} e^{\eps_0 m}.$$
Thus we obtain
\begin{align*}&\sharp\{\d \in \tilde \cR_m(\eps) : e^{2\lambda n}f_{{\e \d}}'(x) \in B(y,e^{-\eps_0 m})\}^2\\
& \lesssim C_{\eps,n}^\kappa \sharp \tilde \cR_m(\eps)^2[ e^{-\delta_1(\Xi) m}  (C_{\varepsilon,m}^{3\lambda} e^{-\eps_0 m} + C_{\varepsilon,m}^{3\lambda} e^{\eps_0 m}) + C_{\varepsilon,m}^{3\lambda} e^{-\eps_0 m}]
\end{align*}
Here we see that the possible obstacle to the decay would come from the term
$$e^{-\delta_1(\Xi) m} e^{\eps_0 m} = e^{-(\delta_1(\Xi) - \eps_0) m}.$$
But since we defined $\eps_0 > 0$ such that $\eps_0 = \min\{\delta_1(\Xi)/2,\lambda/4\}$, recall Remark \ref{rmk:parameters}(2), we obtain $c_0 := \delta_1(\Xi) - \eps_0 \geq \delta_1(\Xi)/2 > 0$. This completes the proof of the estimate \eqref{eq:scalem} and thus the whole lemma.
\end{proof}

Now we can prove Proposition \ref{lma:distribution}

\begin{proof}[Proof of Proposition \ref{lma:distribution}] Recall that 
$$J_n(\eps) = \{ \eta \in \mathbb{R}:e^{\eps_0 n/2} \leq |\eta| \leq C_{\varepsilon,n} e^{\eps_0 n}  \}$$
Fixing $\eta \in J_n(\eps)$ and $\sigma \in  [R^{-2}|\eta|^{-1},|\eta|^{-\eps_3}]$, there is a unique $l$ such that $2^{-l-1} < \sigma \leq 2^{-l}$. Define $\mathcal{R}_l^*$ to be the set of $n$-regular pairs $(\a,\d) \in \mathcal{R}_n(\eps)^2$ such that
	$$ e^{-2\lambda s n} \sharp \{ (\b,\c) \in \cR_n(\eps)^2 : |f'_{\a\b}(x_{\d})-f_{\a\c}'(x_{\d})|\leq e^{-2\lambda n} 2^{-l} \} \leq 2^{-(l+1)c_0/2}. $$
	 In this setting, if we have a block $\A$ such that $(\a_{j-1},\a_j)\in \mathcal{R}_l^*$ for every $j=1,...,k$ and every $l$, then by definition of $\mathcal{R}_l^*$ and by definition of $\zeta_{j,{\a}}(\b)$ we have that
	\begin{align*}
	& e^{-2\lambda s n} \sharp \{ (\b,\c) \in \mathcal{R}_n(\eps)^2: |\zeta_{j,{\a}}({\b})-\zeta_{j,{\a}}({\c})| \leq \sigma \} \\
	&  \leq e^{-2\lambda s n} \sharp \{ (\b,\c) \in \mathcal{R}_n(\eps)^2: |f_{\a_{j-1}\b}'(\a_j)-f_{\a_{j-1}\c}'(\a_j)| \leq e^{-2\lambda n} 2^{-l} \} \\
	& \leq 2^{-(l+1)c_0/2} \\
	& \leq \sigma^{c_0/2}.
	\end{align*}
	This therefore tells us that
	$$ \bigcap\limits_{j}\bigcap\limits_{l}\{ \A : (\a_{j-1},\a_j) \in \mathcal{R}_l^* \} \subset \mathcal{W}. $$
	From this containment, we can say that a $k+1$ block $\A$ is not in $\mathcal{W}$ if there exists at least one position $j$ in the block and a scale $l$ such that the pair $(\a_{j-1},\a_j)\notin \mathcal{R}^*$. So to prove the lemma, it is enough to show that $$e^{-2\lambda \delta n} \sharp( \mathcal{R}_n(\eps)^2\setminus \mathcal{R}^*_l ) \lesssim C_{\varepsilon,n}^{\kappa_0}\sigma^{c_0/2}.$$
	We achieve this bound by considering the counting measure $\sharp$ on pairs in $\mathcal{R}_n(\eps)^2$ and use Chebychev's inequality to get an upper bound on $|\mathcal{R}_n(\eps)^2 \setminus \mathcal{R}^*_l|$. We apply Chebychev's inequality to the counting function defined by 
	$$f(\a,\d)=\ \frac{\sharp \{ (\b,\c) \in \cR_n(\eps)^2 : |f'_{\a\b}(x_{\d}) - f'_{\a\c}(x_{\d})|\leq e^{-2\lambda n}2^{-l} \}}{\sharp \cR_n(\eps)^2}$$ 
	which gives us that
	\begin{align*}
	\sharp\{ \mathcal{R}_n(\eps)^2 \setminus \mathcal{R}_l^* \} & = \sharp\{ (\a,\d)\in \mathcal{R}_n(\eps)^2 : f(\a,\d) \geq 2^{-(l+1)c_0/2} \}\\&\leq 2^{(l+1)c_0/2}\sum_{(\a,\d) \in \mathcal{R}_n(\eps)^2} f(\a,\d)\\
	& = e^{-2\lambda \delta n}2^{(l+1)c_0/2}\sharp\{ (\a,\b,\c,\d) \in \cR_n^4: |f'_{\a\b}(x_{\d})-f'_{\a\c}(x_{\d})| \leq e^{-2\lambda n} 2^{-l} \}
	\end{align*}
	Then by Lemma \ref{thm:nonconc} with $\sigma = 2^{-l}$, we have
	$$e^{-2\lambda \delta n}2^{(l+1)c_0/2}\sharp\{ (\a,\b,\c,\d) \in \cR_n^4(\eps): |f'_{\a\b}(x_{\d})-f'_{\a\c}(x_{\d})| \leq e^{-2\lambda n} 2^{-l} \} \lesssim C_{\varepsilon,n}^{\kappa_0} e^{-\ell c_0/2} $$
	which gives the claim since $e^{-\ell c_0/2} \leq \sigma^{c_0/2}$.
\end{proof}

\section{Proof of Theorem \ref{thm:nonlinear}}\label{sec:completion}

We begin the proof by fixing first $\eps > 0$ small enough that $C_{\eps,n}^\kappa = e^{\kappa \eps n}$ have the exponent $\kappa \eps$ small enough in terms of Lyapunov exponent $\lambda = \int \tau \, d\mu > 0$, Hausdorff dimension $\delta = \Hd \mu > 0$ and the non-concentration parameters $\eps_0,c_0 > 0$. Recall that $\eps_0$ and $c_0$ are fixed in Remark \ref{rmk:parameters} and Proposition \ref{lma:distribution} and they only depend on the spectral gap of $\cL_{\phi - (\delta - 2\pi i )\tau}$ obtained from Theorem \ref{thm:c1cont} in the Banach space $C^{\mathrm{Lip}}(I)$.

To be able to apply relevant large deviations (Theorem \ref{thm:largedev}), we need to make sure that the values of $n$ that we consider are sufficiently large. We begin by choosing $n_0(\eps)$. Recall the assumptions (1), (2) and (3) of the growth and bounded variations mapping $T : I \to \R$ in the introduction.
\begin{enumerate}
	\item If $n_1(\eps)$ is the generation that arises from the main large deviation Theorem \ref{thm:largedev}, then we require
	$$ n_0(\eps) \varepsilon_0 > n_1(\eps) $$
	to ensure we have valid regularity at each scale that we need.
	\item If $\gamma$ is the rate of expansion of $(T^n)'$ with respect to $n$, and $C > 0$ is the Gibbs constant for $\mu$ (recall \eqref{eq:Gibbs}), we require 
	$$\frac{\log 4}{\varepsilon_0 n_0}<\varepsilon_2, \,\,\,\,\, \frac{\log 4C^2}{\log(\gamma^{2 \varepsilon_0 n_0})}<\varepsilon /2 \text{\,\,\,\,\,\and\,\,\,\,\,} \frac{e^{-\delta \varepsilon_0 n_0}}{1-e^{-\delta}} < e^{-\delta \varepsilon_0 n_0/2}$$
	to ensure that we get decay on multiregular blocks of words.
\end{enumerate}

Now, from the main sum-product estimate (Lemma \ref{lma:bdexponential}), fix the parameters $k \in \N$ and $\eps_2 > 0$, which we note depend on the non-concentration parameter $\eps_0 > 0$ fixed in Remark \ref{rmk:parameters}(2).

Fix a frequency $\xi \in \R$ such that $|\xi|$ is large enough. Let $n \in \N$ be the number such that
$$e^{(2k+1) n \lambda} e^{\eps_0 n} \leq |\xi| \leq e^{(2k+1) (n+1) \lambda} e^{\eps_0 (n+1)}.$$
so up to a multiplicative constant depending on $k$ and $\eps_0$, we have:
$$|\xi| \sim e^{(2k+1) n \lambda} e^{\eps_0 n},$$
where $|\xi| \sim N$ means that there exists a constant $c > 0$ such that $c^{-1} N \leq |\xi| \leq cN$. Recall that the cardinality
$$\sharp \cR_n(\eps) \lesssim C_{\varepsilon,n}^{3\lambda} e^{-\lambda \delta n}$$
and if $\a \in \cR_n(\eps)$, we have
$$w_{\a}(x) = e^{S_n \phi(f_\a(x))} \leq C_{\varepsilon,n}^{3\lambda}e^{-\lambda \delta n}$$
for all $x \in [0,1]$, see Lemma \ref{lma:regularmarkov} on the regular words.

We begin by recalling the estimate from Proposition \ref{lma:exponentialsum}. Recall that there we have
$$J_n(\eps) = \{ \eta \in \mathbb{R}:e^{\eps_0 n/2} \leq |\eta| \leq C_{\varepsilon,n} e^{\eps_0 n}  \}.$$ 
and we have the following estimate in terms of exponential sums and error terms:
	\begin{align*} | \widehat{\mu}(\xi)| ^2 \lesssim_\mu \,\,&C_{\varepsilon,n}^{(2k+1)\lambda}e^{-\lambda (2k+1)\delta n} \sum\limits_{{\A} \in \cR_n^{k+1}(\eps)} \sup\limits_{\eta \in J_n(\eps)} \Big| \sum\limits_{{\B \in \cR_n^k(\eps)}} e^{2\pi i \eta \zeta_{1,{\A}}({\bf b_1})...\zeta_{k,{\A}}({\b}_k)} \Big|  \\
	& + e^{2k}C_{\varepsilon,n}^{k+2} e^{-\lambda n}e^{\eps_0 n} +\mu([0,1]\setminus R_n^{k+1}(\eps))^2+ C_{\varepsilon,n}^{4k\lambda} \rho^{2n} + \mu(I\setminus R_n(\eps))+C_{\varepsilon,n}^2e^{- \delta \eps_0 n/2},
	\end{align*}
	where if blocks $\A \in \cR_n^{k+1}(\eps), \B \in \cR_n^{k}(\eps)$, $j \in \{1,\dots,k\}$ and ${\b} \in \mathcal{R}_n$, we defined
$$ \zeta_{j,{\A}}({\b}) = e^{2\lambda n}f_{{\bf a_{j-1}b}}'(x_{{\bf a_j}}). $$

Now, recall that in Proposition \ref{lma:distribution}, we defined the well-distributed blocks of words $\A \in \cW$ as follows: for all $j=1,\dots,k$, $\eta \in J_n(\eps)$ and $\sigma \in [R^2|\eta|^{-1},|\eta|^{-\eps_3}]$, where we have that
	$$ \sharp \{ ({\bf b, c}) \in \mathcal{R}_n(\eps)^2: |\zeta_{j,{\A}}({\b})-\zeta_{j,{\A}}({\c})| \leq \sigma \} \leq \sharp \cR_n(\eps)^2 \sigma^{c_0/2}. $$
Proposition \ref{lma:distribution} said most blocks are well-distributed: there exists $\kappa_0>0$,
	$$ e^{-\lambda (k+1)\delta n} \sharp (\mathcal{R}_n^{k+1}(\eps)\setminus \mathcal{W}) \leq C_{\varepsilon,n}^{2\kappa_0} \sigma^{c_0/4}.$$
For the exponential sum term
$$C_{\varepsilon,n}^{(2k+1)\lambda}e^{-\lambda (2k+1)\delta n} \sum\limits_{{\A} \in \cR_n^{k+1}(\eps)} \sup\limits_{\eta \in J_n(\eps)} \Big| \sum\limits_{{\B \in \cR_n^k(\eps)}} e^{2\pi i \eta \zeta_{1,{\A}}({\bf b_1})...\zeta_{k,{\A}}({\b}_k)} \Big|$$
in the estimate for $|\widehat{\mu}(\xi)|$ we begin by removing not well-distributed blocks, which gives
\begin{align*}
& C_{\varepsilon,n}^{(2k+1)\lambda}e^{-(2k+1)\lambda \delta n} \sum\limits_{{\a} \in \cR_n^{k+1} \setminus \mathcal{W}} \sup\limits_{\eta \in J_n(\eps)} \Big| \sum\limits_{{\B \in \cR_n^k(\eps)}} e^{2\pi i \eta \zeta_{1,{\A}}({\bf b_1})...\zeta_{k,{\A}}({\b}_k)} \Big|\\
& \leq C_{\varepsilon,n}^{(2k+1)\lambda}e^{-(2k+1)\lambda \delta n} \sum\limits_{{\A} \in \cR_n^{k+1}\setminus \mathcal{W}} \sup\limits_{\eta \in J_n(\eps)} \sum\limits_{{\B \in \cR_n^k(\eps)}} 1\\
& \leq C_{\varepsilon,n}^{(2k+1)\lambda}e^{-(2k+1)\lambda \delta n} \sum\limits_{{\A} \in \cR_n^{k+1}\setminus \mathcal{W}} C^kC_{\varepsilon,n}^{3\lambda k}e^{k\lambda \delta n} \\
&\lesssim C^kC_{\varepsilon,n}^{(5k+7)\lambda} e^{-(k+1)\lambda \delta n}e^{(k+1)\lambda \delta n} \sigma^{c_0/4} \\
& \lesssim C_{\varepsilon,n}^{\kappa_0'} \sigma^{c_0/4}
\end{align*}
where $\kappa_0'>0$. 
Hence we have that
\begin{align*}  | \widehat{\mu}(\xi)| ^2 \lesssim\,\, & C_{\varepsilon,n}^{k\lambda}e^{-k\lambda \delta n} \max_{{\bf A \in \cW}} \sup\limits_{\eta \in J_n(\eps)} \Big| \sum\limits_{{\B \in \cR_n^k(\eps)}} e^{2\pi i \eta \zeta_{1,{\A}}({\bf b_1})...\zeta_{k,{\A}}({\b}_k)} \Big|\\
& + C_{\varepsilon,n}^{\kappa_0'} \sigma^{c_0/4} + e^{2k}C_{\varepsilon,n}^{k+2} e^{-\lambda n}e^{\eps_0 n} +\mu([0,1]\setminus R_n^{k+1}(\eps))^2\\
& + C_{\varepsilon,n}^{4k\lambda} \rho^{2n} + \mu([0,1] \setminus R_n(\eps))+C_{\varepsilon,n}^2e^{- \delta \eps_0 n/2}.
\end{align*}
Recall Remark \ref{rmk:parameters}, where we defined $R := 16^2CC_{\varepsilon,n}^{3\lambda}$. Thus $\zeta_{j,\A}(\b) \in [R^{-1},R]$. Moreover, if we fix $\eta \in J_n(\eps)$, $A \in \cW$ and $\sigma \in [R^{-2}|\eta|^{-1},|\eta|^{-\eps_3}]$, we have by the definition of $\cW$ that
$$ \sharp \{ (\b,\c)\in \mathcal{R}_n(\eps)^2 : |\zeta_{j,\A}(\b) - \zeta_{j,\A}(\c)| \leq \sigma\} \leq \sharp \cR_n(\eps)^2 \sigma^{c_0/2}$$
Thus we may apply Lemma \ref{lma:newexponential} to the maps $\zeta_{j,\A} : \cR_n(\eps) \to [R^{-1},R]$ with $N = e^{\lambda \delta n}$. It implies that for all $\A \in \cW$ and $\eta \in J_n(\eps)$ that
$$C_{\varepsilon,n}^{k\lambda}e^{-k\lambda \delta n}\Big| \sum\limits_{{\B \in \cR_n^k(\eps)}} e^{2\pi i \eta \zeta_{1,{\a}}({\B_1})...\zeta_{k,{\a}}({\b}_k)} \Big| \lesssim R^{2k\lambda} |\eta|^{-\eps_2} \lesssim C_{\varepsilon,n}^{9k\lambda}e^{-\eps_0 \eps_2 n/2}$$
since $|\eta| \geq C_{\varepsilon,n}^{-1}e^{\eps_0 n/2}$ by  the definition of $J_n(\eps)$ as $\eta \in J_n(\eps)$. By making sure that $\eps > 0$ is chosen small enough and as $\eps_0 \leq \lambda/2 < \lambda$, recall Remark \ref{rmk:parameters}(2) for the choice of $\eps_0$ using the spectral gap of the transfer operator, which is independent of $\xi$, we have proved
$$|\widehat{\mu}(\xi)| = O(|\xi|^{-\alpha})$$
as $|\xi| \to \infty$ for some $\alpha > 0$. The proof of Theorem \ref{thm:nonlinear} is complete.

\section{Acknowledgements}

We thank Semyon Dyatlov for discussions related to the problem that led to this work. The proof of the non-concentration of the derivatives is based on discussions TS had with Jialun Li in Bordeaux and Oxford in 2019, which in turn were based on discussions Jialun Li had with Fr\'ed\'eric Naud in Avignon in 2018. We thank Jialun Li and Fr\'ed\'eric Naud that we could include the ideas to this work. We also thank Jialun Li for pointing out the work of Stoyanov that generalised the work of Naud and allowed us to drop a condition that the dimension has to be close enough to the dimension of the attractor. We also thank Michael Magee for useful discussions on total non-linearity and Non-Local Integrability. Furthermore we thank Demi Allen, Xiong Jin, Thomas Jordan, Tomas Persson and Charles Walkden for useful discussions during the preparation of this manuscript. We thank the anonymous referee for comments on a previous version of the manuscript.

\bibliographystyle{plain}

\begin{thebibliography}{10}

\bibitem{AGY}
A. Avila, S. Gou\"ezel, J. Yoccoz. Exponential mixing for the Teichm\"uller flow. \textit{Publ. Math. IHES} 104, 143–211 (2006).

\bibitem{BV}
V. Baladi, B. Vall\'ee. Exponential Decay of Correlations for Surface Semi-Flows without Finite Markov Partitions. \textit{Proc. Amer. Math. Soc.}, Vol. 133, No. 3, pp. 865-874 (2005)

\bibitem{BenoistSaxce}
Y. Benoist, N. Saxc\'e. A spectral gap theorem in simple Lie groups. \textit{Invent. Math.} 205, 337--361 (2016).
	
\bibitem{Bourgain2010}
J. Bourgain. The discretized sum-product and projection theorems, \textit{J. Anal. Math.} 112(2010), 193–236.
	
\bibitem{bourg}
J. Bourgain and S. Dyatlov.
\newblock Fourier dimension and spectral gaps for hyperbolic surfaces.
\newblock \textit{GAFA} 27(2017), 744–771

\bibitem{BourgainDyatlov}
J. Bourgain and S. Dyatlov.
Spectral gaps without the pressure condition, \textit{Ann. of Math.} 187(2018), 825–867 

\bibitem{BourgainGamburd}
J. Bourgain, A. Gamburd: A Spectral Gap Theorem in $SU(d)$, \textit{ JEMS} 14.5 (2012): 1455-1511

\bibitem{BowenSeries}
R. Bowen, C. Series: Markov maps associated with Fuchsian groups, \textit{IHES Publications} 50, 1979, 153 – 170.

\bibitem{Bremont}
 J. Br\'emont. Self-similar measures and the Rajchman property. arXiv:1910.03463, Preprint (2019)
 
\bibitem{Chernov}
N. Chernov, Markov approximations and decay of correlations for Anosov flows, \textit{Ann. of Math.} 147, 269-324 (1998)
 
\bibitem{Cohen}
P. Cohen. Topics in the theory of uniqueness of trigonometrical series, PhD thesis, 1958, \url{http://www.lix.polytechnique.fr/Labo/Ilan.Vardi/cohen.ps}

\bibitem{CyrSarig}
V. Cyr, O. Sarig: Spectral Gap and Transience for Ruelle Operators on Countable Markov Shifts. \textit{Comm. Math. Phys.}, 292, pp. 637–666(2009)

\bibitem{DEL}
H. Davenport, P. Erd\"{o}s, and W. LeVeque.
\newblock On Weyl’s criterion for uniform distribution.
\newblock \textit{Michigan Math. J.}, 10:311-314, 1963.

\bibitem{Dolgopyat1} 
D. Dolgopyat, On decay of correlations in Anosov flows. \textit{Ann. of Math.} (2) 147(1998), 357-390.

\bibitem{Dolgopyat}
D. Dolgopyat. Prevalence of rapid mixing in hyperbolic flows. \textit{ETDS}. Volume 18, Issue 5 October 1998, pp. 1097-1114

\bibitem{DJ17}
S. Dyatlov, L. Jin: 
\newblock Semiclassical measures on hyperbolic surfaces have full support.
\newblock \textit{Acta Math.}, Volume 220, Number 2 (2018), 297-339.

\bibitem{DyatlovZahl}
S. Dyatlov, J. Zahl: Spectral gaps, additive energy, and a fractal uncertainty principle, \textit{GAFA} 26(2016), 1011–1094

\bibitem{Dyatlov} S. Dyatlov: An introduction to fractal uncertainty principle, \textit{Journal of Math. Phys.} 60(2019), 081505

\bibitem{EkstromSchmeling} F. Ekstr\"om, J. Schmeling. A Survey on the Fourier Dimension. International Conference on Patterns of Dynamics PaDy 2016: Patterns of Dynamics, 67--87

\bibitem{Erdos} P. Erd\"os. On the smoothness properties of a family of Bernoulli convolutions. \textit{Amer. J. Math.}, 62:180–186, 1940.

\bibitem{HochmanShmerkin}
M. Hochman, P. Shmerkin. Equidistribution from fractal measures. \textit{Invent Math.} Volume 202, Issue 1, pp 427–479, 2015.

\bibitem{Jialun} J. Li: Decrease of Fourier Coefficients of Stationary Measures, \textit{Mathematische Annalen}, December 2018, Volume 372, Issue 3--4, pp 1189--1238.

\bibitem{JialunRd} J. Li: Discretized Sum-product and Fourier decay in $\R^n$, 2018, to appear in \textit{Journal d’Analyse Math.}

\bibitem{JialunNaudPan} J. Li, F. Naud, W. Pan: Kleinian Schottky groups, Patterson-Sullivan measures, and Fourier decay, with an appendix on stationarity of Patterson-Sullivan measures, to appear in \textit{Duke Math. J.}

\bibitem{JialunSahlsten1} J. Li, T. Sahlsten: Trigonometric Series and Self-similar Sets, \textit{JEMS}, to appear, arXiv:1902.00426, 2021

\bibitem{JialunSahlsten2} J. Li, T. Sahlsten: Fourier transform of self-affine measures, \textit{Adv. Math.} (2020), Volume 374, 107349

\bibitem{JordanSahlsten}
T. Jordan and T. Sahlsten.
\newblock Fourier transforms of Gibbs measures for the Gauss map.
\newblock {\em Math. Ann. (2016) Vol 364 (3)}. 983-1023, 2015.

\bibitem{KahaneImage}
J.-P. Kahane.
\newblock  Some Random series of functions (2nd ed.).
\newblock {\em  Cambridge University Press}. 1985.

\bibitem{KahaneLevel}
J.-P. Kahane.
\newblock  Ensembles alatoires et dimensions.
\newblock {\em  In Recent Progress in Fourier Analysis (proceedings of a seminar held in El Escorial). North-Holland, Amsterdam}. 65-121, 1983.

\bibitem{Kaufman}
R. Kaufman.
\newblock  Continued fractions and Fourier transforms.
\newblock {\em  Mathematika, 27(2)}. 262-267, 1980.

\bibitem{Kauf2overalpha}
R. Kaufman.
\newblock On the theorem of Jarn\'{i}k and Besicovitch.
\newblock {\em  Acta Arith.}. 39(3):265-267, 1981.

\bibitem{KaufmanSurvey}
R.  Kaufman. Robert Kaufman. On Bernoulli convolutions. In Conference in modern analysis and probability, vol. 26 of Contemp. Math., pages 217--222. Amer. Math. Soc., Providence, RI, 1984.

\bibitem{K}
A. Kechris and A. Louveau. Descriptive set theory and the structure of sets of uniqueness (London Mathematical Society lecture series 128), Cambridge University Press. ISBN 0-521-35811-6, 1987

\bibitem{Keller}
G. Keller. \textit{Equilibrium States in Ergodic Theory}, 1998. Cambridge University Press

\bibitem{LV} E. Lindenstrauss, P. Varj\'u: Random walks in the group of Euclidean isometries and self-similar measures. \textit{Duke Math. J.}. Volume 165, Number 6 (2016), 1061--1127.

\bibitem{Lyons}
R. Lyons. Seventy Years of Rajchman Measures. \textit{Journal of Fourier Analysis and Applications Special Issue}. CRC Press, 21 Sep 1995.

\bibitem{MOW} M. Magee, H. Oh, D. Winter: Uniform congruence counting for Schottky semigroups in $SL_2(\Z)$. 
\textit{Crelle's Journal}, Volume 2019: Issue 753

\bibitem{Mattila95}
P. Mattila.
\newblock Geometry of Sets and Measures in Euclidean Spaces: Fractals and Rectifiability.
\newblock {\em Cambridge University Press, The Edinburgh Building, Cambridge CB2 2RU, UK}. Published 1995.

\bibitem{Mattila15}
P. Mattila.
\newblock Fourier Analysis and Hausdorff Dimension.
\newblock {\em Cambridge University Press, University Printing House, Cambridge CB2 8BS, UK}. 2015.

\bibitem{MS}
C. Mosquera, P. Shmerkin: Self-similar measures: asymptotic bounds for the dimension and Fourier decay of smooth images. \textit{Ann. Acad. Sci. Fenn. Math.}. 2018, Vol. 43 Issue 2, p823-834

\bibitem{Naud}
F. Naud: Expanding maps on Cantor sets and analytic continuation of zeta functions. \textit{Ann. ENS}, Serie 4, Volume 38 (2005) no. 1, p. 116-153

\bibitem{PRKS} Y. Peres, M. Rams, K. Simon and B. Solomyak: Equivalence of Positive Hausdorff Measure and the Open Set Condition for Self-Conformal Sets, \textit{Proc. Amer. Math. Soc.} Vol. 129, No. 9 (Sep., 2001), pp. 2689-2699 (11 pages)

\bibitem{PS}
M. Pollicott, R. Sharp. Large deviations, fluctuations and shrinking intervals. \textit{Comm. Math. Phys}, volume 290, nr 321 (2009)

\bibitem{PVZZ} A. Pollington, S. Velani, A. Zafeiropoulos, E. Zorin: Inhomogeneous Diophantine Approximation on $M_0$-sets with restricted denominators, arXiv:1906.01151, \textit{Preprint} (2019)

\bibitem{QetR}
M. Queff\'{e}lec and O. Ramar\'{e}.
\newblock Analyse de Fourier des fractions continues \'{a} quotients restreints.
\newblock \textit{Enseign. Math.} (2), 49(3-4):335-356, 2003.

\bibitem{Ruelle} D. Ruelle, A review of linear response theory for general differentiable dynamical systems, \textit{Nonlinearity} 22 (2009) 855--870.

\bibitem{Salem}
R. Salem.
\newblock On some singular monotonic functions which are strictly increasing.
\newblock {\em  Trans. Amer. Math. Soc}. 53:427-439, 1943.

\bibitem{Salem2} R. Salem: Sets of uniqueness and sets of multiplicity. \textit{Trans AMS}, 54:218-228, 1943. Corrected on pp. 595-598, Trans AMS, vol 63, 1948.

\bibitem{SZ} R. Salem, A. Zygmund. Sur un theoreme de Piatetski-Shapiro. \textit{CRASP}, 240:2040-2042, 1954.

\bibitem{Sarig}
O. Sarig.
\newblock  Existence of Gibbs measures for countable Markov shifts.
\newblock {\em  Proc. Amer. Math. Soc., 131(6)}. 1751-1758, 2003.

\bibitem{Sarnak} P. Sarnak, Spectra of singular measures as multipliers on $L^p$, \textit{J. Funct. Anal.} 37 (1980), 302--317.

\bibitem{ShmerkinSuomala}
P. Shmerkin and V. Suomala: Spatially independent martingales, intersections, and applications. \textit{Mem. Amer. Math. Soc.} 251 (2018), no. 1195

\bibitem{SidorovSolomyak}
N. Sidorov, B. Solomyak. Spectra of Bernoulli convolutions as multipliers in $L^p$ on the circle. \textit{Duke Math. J.}
Volume 120, Number 2 (2003), 353-370.

\bibitem{Solomyak} B. Solomyak. Fourier decay for self-similar measures. \textit{Proc. Amer. Math. Soc.} to appear, arXiv:1906.12164, 2021.

\bibitem{Stein} E. Stein. Harmonic analysis on $\R^n$, in Studies in Harmonic Analysis” (J. M. Ash, Ed.), pp. 97-135,
Studies in Mathematics, Vol. 13, Mathematical Association of America, Washington, DC, 1976.

\bibitem{Stoyanov} L. Stoyanov. Spectra of Ruelle transfer operators for Axiom A flows. \textit{Nonlinearity} (2011), Volume 24, Number 4

\bibitem{VarjuYu} P. Varj\'u, H. Yu: Fourier decay of self-similar measures and self-similar sets of uniqueness, \textit{Analysis \& PDE} (2020), to appear, arXiv:2004.09358

\bibitem{Zygmund} A. Zygmund. \textit{Trigonometric Series}, two volumes, Cambridge Univ. Press, Cambridge, England, 1959.

\end{thebibliography}

\vspace{0.1in}

\end{document}